\documentclass{amsart}

\usepackage{amsmath}
\usepackage{amsthm}
\usepackage{amssymb}
\usepackage{amsfonts}
\usepackage{bm}
\usepackage[shortlabels]{enumitem}
\usepackage[bookmarks=true,hyperindex,pdftex,colorlinks,citecolor=red, linkcolor=cyan]{hyperref}
\usepackage{centernot} %for negations
\usepackage{marginnote} % add margin notes
\usepackage{bbm}
\usepackage[all]{xy}
\usepackage{tikz}
\usetikzlibrary{arrows}
\usetikzlibrary{shapes,decorations}
\usetikzlibrary{positioning}
\usepackage{tikz-cd} 
\usepackage[normalem]{ulem}
\usepackage{enumitem}

\allowdisplaybreaks

\DeclareMathOperator{\dist}{dist}                           % distance between sets
\DeclareMathOperator{\lspan}{span}                          % linear span
\DeclareMathOperator{\conv}{conv}                           % convex hull
\DeclareMathOperator{\supp}{supp}                           % support
\DeclareMathOperator{\diam}{diam}                           % diameter
\DeclareMathOperator{\Lip}{Lip}                             % Lipschitz functions
\newcommand{\N}{\mathbb{N}}             % natural numbers
\newcommand{\M}{\mathbb{M}}             % an infinite subset of natural nubmers
\newcommand{\Z}{\mathbb{Z}}             % integer numbers
\newcommand{\Integer}{\Z}             
\newcommand{\R}{\mathbb{R}}             % real numbers
\newcommand{\Real}{\R}             
\newcommand{\abs}[1]{\left|{#1}\right|}                     % absolute value
\newcommand{\cardinality}[1]{\abs{#1}}
\newcommand{\pare}[1]{\left({#1}\right)}                    % parentheses
\newcommand{\set}[1]{\left\{{#1}\right\}}                   % set by extension
\newcommand{\norm}[1]{\left\|{#1}\right\|}                  % norm
\newcommand{\duality}[1]{\left<{#1}\right>}                 % dual action
\newcommand{\cl}[1]{\overline{#1}}                          % closure
\newcommand{\lipfree}[1]{\mathcal{F}({#1})}                 % Lipschitz free space
\newcommand{\F}{\mathcal{F}}                                % Lipschitz free space
\newcommand{\Free}{\mathcal{F}}                             % Lipschitz free space
\newcommand{\SSM}[2]{\mathrm{CSM}_{#1}\left({#2}\right)}

\newcommand{\restricted}{\mathord{\upharpoonright}}
\renewcommand{\restriction}{\mathord{\upharpoonright}}

\def\<{\langle}
\def\>{\rangle}
\newcommand{\ep}{\varepsilon}
\newcommand{\indicator}[1]{{\mathbf 1}_{{#1}}}
\renewcommand{\hat}{\widehat}

\renewcommand{\leq}{\leqslant}
\renewcommand{\geq}{\geqslant}

\theoremstyle{plain}
\newtheorem{theorem}{Theorem}[section]
\newtheorem{lemma}[theorem]{Lemma}
\newtheorem{corollary}[theorem]{Corollary}
\newtheorem{proposition}[theorem]{Proposition}

\newtheorem{maintheorem}{Theorem} % for fancy main theorems with letter numbering
\newtheorem{maincorollary}[maintheorem]{Corollary} % for fancy main theorems with letter numbering

\theoremstyle{definition}
\newtheorem*{definition*}{Definition}
\newtheorem{definition}[theorem]{Definition}
\newtheorem{example}[theorem]{Example}
\newtheorem{question}{Question}

\theoremstyle{remark}
\newtheorem{remark}[theorem]{Remark}

\begin{document}
	%-------------------------------------------------------
	\title[Weak Compactness and Uniform Regularity in free Spaces]{On Weak Compactness and Uniform Regularity in Lipschitz free Spaces}
	%-------------------------------------------------------
	
	% KEYWORDS 

	%46B50 Compactness in Banach (or normed) spaces
	%46B20 Geometry and structure of normed linear spaces


	\author[R. J. Aliaga]{Ram\'on J. Aliaga}
	\address[R. J. Aliaga]{Instituto Universitario de Matem\'atica Pura y Aplicada, Universitat Polit\`ecnica de Val\`encia, Camino de Vera S/N, 46022 Valencia, Spain}
	\email{ramon.aliaga@upv.es}
	
	\author[C. Petitjean]{Colin Petitjean}
	\address[C. Petitjean]{Univ Gustave Eiffel, Univ Paris Est Creteil, CNRS, LAMA UMR8050, F-77447 Marne-la-Vallée, France}
	\email{colin.petitjean@univ-eiffel.fr}
	
	\author[A. Prochazka]{Anton\'in Prochazka}
	\address[A. Prochazka]{Université Marie et Louis Pasteur, CNRS, LmB (UMR 6623), F-25000 Besançon, France}
	\email{antonin.prochazka@umlp.fr}
	
	\author[T. Veeorg]{Triinu Veeorg}
	\address[T. Veeorg]{University of Tartu, Institute of Mathematics and Statistics, Narva mnt 18, 51009 Tartu, Estonia and University of Innsbruck, Department of Mathematics, Technickerstraße 13, 6020 Innsbruck, Austria (Present Address)}
	\email{triinu.veeorg@uibk.ac.at}

	\begin{abstract}
		We analyze the properties of weakly compact sets in Lipschitz free spaces. Prior research has established that, for a complete metric space $M$, weakly precompact sets in the Lipschitz free space $\mathcal F(M)$ are tight. In this paper,  we prove that these sets actually exhibit a stronger property, which we call uniform regularity. However, this condition alone is not sufficient to characterize weakly compact sets, except in the case of scattered metric spaces. On the other hand, if $T$ is an $\R$-tree, we leverage Godard's isometry between $\mathcal F(T)$ and $L^1(\lambda_T)$ to obtain an intrinsic characterization of weakly compact sets in $\mathcal F(T)$. This approach allows us to identify conditions that may describe weak compactness across a wider range of spaces. In particular, we provide a characterization of norm-compactness in terms of sums of ``large molecules'', while we show that sums of ``small molecules'' contain an $\ell_1$-basis.
	\end{abstract}

	\maketitle

	%----------------------------------------------------------
	\section*{Introduction}

	The weak topology is of the utmost importance when studying the geometry of Banach spaces. In particular, a good knowledge of weakly compact subsets is crucial. Indeed, many Banach space properties are intimately related to weak compactness, which in turn provide very helpful tools to understand, analyze, and classify Banach spaces. In this paper, we study weakly compact sets in the so-called Lipschitz free spaces. In a nutshell, for a pointed metric space $(M,d)$, the Lipschitz free space $\F(M)$ is the free vector space generated by $M$, equipped with a particular norm which makes it a Banach space. The norm is designed in such a way that $M$ is isometric to a linearly dense subset $\delta(M)$ of $\F(M)$, and moreover, Lipschitz maps from $\delta(M)$ into any Banach space $X$ uniquely extend to bounded linear operators from $\F(M)$ into $X$. See Section~\ref{section:LipFree} for more details.
	
	Weak convergence and weakly compact sets in general Lipschitz free spaces are relatively poorly understood.
	Nevertheless, a successful first step was made in \cite{ANPP} by isolating a powerful necessary condition, called \textit{tightness}. Recall that a set of measures $S$ on $M$, equipped with the Borel $\sigma$-algebra, is \textit{tight} if, for any $\ep>0$, there is a compact subset $K \subset M$ such that $|\mu|(M \setminus K) < \ep$ for all measures $\mu \in S$ ($|\mu|$ being the total variation of $\mu$). By analogy to this notion, a bounded subset $W$ in a Lipschitz free space $\F(M)$ is called \textit{tight} if, for any $\ep >0$, there is a compact subset $K \subset M$ such that:
	$$ W \subset \F(K) + \ep B_{\F(M)},$$
	where $B_{\F(M)}$ is the unit ball of $\F(M)$. Roughly speaking, this says that every element of $W$ can be approximated arbitrarily closely by members of a Lipschitz free space over a compact subset $K$ of $M$.
	It is proved in \cite{ANPP} that ``Every weakly compact subset of a free space is tight''. 
	This of course bears a resemblance to Prokhorov's theorem, which asserts that a collection $S$ of probability measures is tight if and only if the closure of $S$ is sequentially compact in the space of probability measures equipped with the topology of weak convergence of measures.
	Unfortunately, tightness in Lipschitz free spaces is far from being sufficient for weak compactness. For instance, for any infinite compact metric space $M$, $B_{\F(M)}$ is tight while obviously not weakly compact (infinite-dimensional free spaces are never reflexive, see e.g. \cite[Corollary 3.46]{Weaver2}). Nevertheless, the discovery of this condition had an immediate effect on the general understanding of linear properties such as the Schur property (SP), weak sequential completeness (wsc), the Radon-Nikod\'ym property (RNP), and the approximation property (AP), to mention a few. See e.g. \cite{AGPP21, ANPP, APQ, curveflat}.
	
	The first half of this paper deals with a strengthening of the tightness condition, called \textit{uniform regularity}.  Let us denote $\mathcal M(K)$ the space of all Radon (i.e. finite and regular) signed Borel measures on a compact Hausdorff space $K$. Recall that a subset $S \subset \mathcal M(K)$ is said to be uniformly regular if given any open set $U \subset K$ and $\ep >0$, there is a compact set $H \subset U$ such that $|\mu(U\setminus H)| < \ep$ for all $\mu \in S$. A famous result due to Grothendieck (see e.g. \cite[Theorem~5.3.2]{AlbiacKalton}) states that a bounded subset $S \subset \mathcal M(K)$ is relatively weakly compact if and only if it is uniformly regular. 
	By analogy, we introduce the following definition:
	
	\begin{definition*}
		Let $M$ be a metric space. We say that a bounded subset $W \subset \F(M)$ is \textit{uniformly regular} if for any open set $U \subset M$ and $\ep >0$, there is a compact set $K \subset U$ such that:
		$$ W \subset \F(K \cup (M \setminus U)) + \ep B_{\F(M)}. $$
	\end{definition*}
	
	It is rather easy to see that uniform regularity is strictly stronger than tightness (see Remark~\ref{ex:ball_not_UR}). 
	Our first goal will be to prove the next result:
	
	\begin{maintheorem} \label{thmA}
		Let $M$ be a complete metric space. If $W\subset\F(M)$ is a V*-set, in particular if $W$ is relatively weakly compact, then it is uniformly regular.
	\end{maintheorem}

	See Section~\ref{section:Defs} for the definition of V*-sets. Unfortunately, uniform regularity alone is not sufficient to characterize weak compactness for a general $M$. However, it does so on every complete scattered metric space.

	\begin{maintheorem} \label{thmScattered} 
		Let $M$ be a complete metric space. The following assertions are equivalent:
		\begin{enumerate}
			\item $M$ is scattered;
			\item Every uniformly regular set in $\F(M)$ is relatively compact;
			\item Every uniformly regular set in $\F(M)$ is relatively weakly compact.
			\item  Every uniformly regular set in $\F(M)$ is a V*-set.
		\end{enumerate}
	\end{maintheorem}
	
	In particular, the above theorem shows that for $M=C$, the middle-thirds Cantor set, there exists a uniformly regular set that is not relatively weakly compact. This is notable because $\F(C)$ enjoys the Schur property (being isometric to $\ell_1$, see e.g. \cite{Godard}) and so weak compactness and norm compactness are actually equivalent in this setting. It also yields the following corollary, which appears to be new.
	
	\begin{maincorollary}\label{c:CorollaryC}
		If $M$ is a complete scattered metric space, then $\F(M)$ has property (V*); that is, every V*-set in $\F(M)$ is relatively weakly compact. 
	\end{maincorollary}
	
	\medskip
	
	In the second half of the paper we tackle the problem from a different angle. Recall that an $\R$-tree is an arc-connected metric space $(T, d)$ with the property that there is a unique arc connecting any pair of points $x \neq y \in T$ and it is moreover isometric to the real segment $[0, d(x, y)] \subset \R$. In such a tree $T$, there is a canonical measure $\lambda_T$, called the length measure, which provides a natural way to measure the size of subsets, similar to Lebesgue measure on the real line. Godard proved in \cite{Godard} that $\F(T)$ can be naturally identified with the space $L^1(\lambda_T)$. Now one should keep in mind that relatively weakly compact sets in $L^1(\lambda_T)$ are well understood: they are precisely the uniformly integrable subsets of $L^1(\lambda_T)$. Therefore we will transfer the equi-integrability conditions through the isometry $L^1(\lambda_T) \to \F(T)$ and describe the resulting properties in $\F(T)$. This will provide us some new insights on the properties of weakly compact sets. More precisely, this highlights the importance of representation of elements in $W$ as combinations of elementary molecules satisfying specific properties. We remind the reader that every element $\gamma$ in $\F(M)$ can be written as a series $\gamma = \sum_{i=1}^{\infty} a_i m_{x_i y_i}$ where $(a_i)_i \in \ell_1$ and $m_{x_i y_i} = d(x_i , y_i)^{-1}(\delta(x_i) - \delta(y_i))$ are molecules (see \cite[Lemma~2.1]{AP20} or \cite[Lemma~3.3]{APPP}). If $\|\gamma \| = \sum_{i=1}^{\infty} |a_i|$, such a representation is then referred to as \textit{convex series of molecules}. 
	
	\begin{maintheorem} \label{thmC}
		Let $T$ be an $\R$-tree. If $W \subset \F(T)$ then $W$ is relatively weakly compact if and only if it is bounded and satisfies these two properties: 
		\begin{enumerate}[$(a)$,leftmargin=*]
			\item $\forall \ep >0$, $\exists \delta >0$ such that for all $w \in W$ and for every convex series of molecules $\sum_{i=1}^n a_i m_{x_i  y_i} \in \F(T)$ with $\|w - \sum_i a_i m_{x_i y_i}\| \leq \ep$, $\forall I \subset \{1, \ldots , n\}$, $\sum_I d(x_i ,y_i) < \delta$ implies that $\sum_I |a_i| < 2\ep$.
			\item For any $\ep >0$,  there exists a finitely generated subtree $T':=\bigcup_{i=1}^n [0,z_i] \subset T$ such that $W \subset \F(T') + \ep B_{\F(T)}$.
		\end{enumerate}    
	\end{maintheorem}
	
	Informally, property (a) says that the weight of molecules formed by pairs of points in close proximity should be controlled when approximating elements of a weakly compact set by convex series of molecules. Our next result shows that this is somehow necessary in general.

	\begin{maintheorem} \label{thmD}
		Let $M$ be a complete metric space. For every $n \in \N$, let  $\gamma_n=\sum_k \lambda_k^n m_{x_k^n y_k^n} \in S_{\F(M)}$ be such that $\sum_k d(x_k^n,y_k^n)\to 0$ and $\sup_n\sum_k \lambda_k^n<\infty$.
		Then $(\gamma_n)_{n}$ has a subsequence which is an (asymptotically isometric) $\ell_1$-basis.
	\end{maintheorem}

	Clearly, if a set $W$ contains a sequence $(\gamma_n)$ as above, then $W$ is not relatively weakly compact.

	\section{Preliminaries}

	\subsection{Notation}
	
	The letter $X$ is reserved for a real Banach space. As usual, $X^*$ denotes the topological dual of $X$, while $B_X$ and $S_X$ stand for the unit ball and the unit sphere of $X$, respectively. If $S \subset X$ then $\conv(S)$ is the convex hull of~$S$. 
	Ordinarily, $L^1 = L^1([0,1])$ represents the Banach space of Lebesgue integrable functions from $[0,1]$ to $\R$. 
	Lebesgue measure in $\R$ is denoted by $\lambda$. Similarly, $\ell_1 = \ell_1(\N)$ is the space of absolutely summable sequences in $\R$ and $c_0$ is the space of all real sequences which converge to 0.
	To indicate that $X$ is isometrically isomorphic to another Banach space $Y$, we use the notation $X \equiv Y$.

	Throughout the paper, $(M,d)$ is a complete and pointed metric space, with distinguished point denoted by $0$.
	If $p \in M$ and $r\geq 0$, then we will use the notation:
	\begin{align*}
		B(p,r) &=  \{x \in M \; | \; d(x,p) \leq r \}\\
		B^O(p,r) &=  \{x \in M \; | \; d(x,p) < r \}
	\end{align*}
	Moreover, for $S \subset M$ we write $S^c:=M\setminus S$ and 
	\begin{align*}
		[S]_r &= \set{x\in M:d(x,S)\leq r}\\
		\widetilde{S} &= \{(x,y) \in S \times S \; : \; x \neq y\}.
	\end{align*}
	
	If $M,N$ are pointed metric spaces, with both base points denoted by $0$, then $\Lip_0(M,N)$ stands for the set of all Lipschitz functions $f :M \to N$ such that $f(0) = 0$. As usual, $\Lip(f)$ refers to the best Lipschitz constant of $f$. If $N=X$ is a real Banach space, then $\Lip_0(M,X)$ naturally becomes a Banach space when equipped with the Lipschitz norm:
	$$ \forall f \in \Lip_0(M,X), \quad \|f\|_L := \Lip(f) =\sup_{x \neq y \in M} \frac{\|f(x)-f(y)\|_X}{d(x,y)}.$$
	In the case $X = \R$, it is customary to write $\Lip_0(M):=\Lip_0(M,\R)$.

	\subsection{Lipschitz free spaces}
	\label{section:LipFree}
	
	For $x\in M$, we let $\delta(x) \in \Lip_0(M)^*$ be the evaluation functional defined by $\langle\delta(x) , f \rangle  = f(x), \ \forall f\in \Lip_0(M)$. It is readily seen that $\delta : x \in M \mapsto \delta(x) \in \Lip_0(M)^*$ is an isometry. The \textit{Lipschitz free space over $M$} is then defined as the norm-closure of the space generated by such evaluation functionals:
	$$\F(M) := \overline{ \mbox{span}}^{\| \cdot  \|}\left \{ \delta(x) \, : \, x \in M  \right \} \subset \Lip_0(M)^*.$$
	It is well known that $\F(M)$ is actually a predual of $\Lip_0(M)$, that is $\Lip_0(M) \equiv \F(M)^*$. 
	For an in-depth understanding of Lipschitz free spaces, interested readers are referred to \cite{Weaver2} (where they are called \textit{Arens-Eells spaces}).
	
	Let us recall some important features of Lipschitz free spaces. 
	Firstly, note that, on bounded subsets of $\Lip_0(M)$, the weak$^*$ topology induced by $\F(M)$ coincides with the pointwise topology.
	If $N \subset M$ is closed, then $\F(N\cup\set{0}) \equiv \F_M(N)$ where: 
	$$\F_M(N) := \overline{\text{span}} \{ \delta_M(x) \; | \; x \in N\} \subset \F(M).$$
	Under this identification, the \textit{support of $\gamma \in \F(M)$} is the smallest closed subset $N \subset M$ such that $\gamma \in \F_M(N\cup\set{0})$. 
	It is denoted by $\text{supp}(\gamma)$. The set of finitely supported elements of $\F(M)$ is precisely $\mathrm{span}\,\delta(M)$, so it is dense in $\F(M)$. For further details, we refer to \cite{AP20, APPP}.
	
	Next, let $ f : M \to N $ be a Lipschitz map preserving the basepoint. 
	By the universal linearization property of Lipschitz-free spaces, there exists a unique bounded linear operator 
	\[
	\widehat{f} : \F(M) \to \F(N)
	\]
	satisfying $\widehat{f} \circ \delta_M = \delta_N \circ f$. 
	Moreover, $\|\widehat{f}\| = \Lip(f)$.

	If $x \neq y  \in M$, then the \textit{elementary molecule} $m_{xy} \in S_{\F(M)}$ is defined by:
	$$ m_{x y} := \frac{\delta(x) - \delta(y)}{d(x,y)}.$$
	It is well-known and easy to prove that unit ball of $\F(M)$ is the closed convex hull of the set of all elementary molecules. Moreover if $Q:\ell_1(\widetilde{M})\to \mathcal F(M)$ denotes the linear map such that 
	$$ Q(e_{(x,y)}) = m_{xy},$$
	then $Q$ extends to an onto norm-one mapping which satisfies 
	$$(\star) \qquad \forall \mu \in \lipfree{M}, \quad \norm{\mu}=\inf\set{\norm{a}_1: Qa=\mu}.$$ In other words, $Q$ is a quotient map (see e.g. Lemma 3.3 in \cite{APPP}). Is it important to note that the infimum in $(\star)$ is not always attained. However, when it is, that is if there is $a = (a_{(x,y)})_{x,y \in \widetilde{M}} \in \ell_1(\widetilde{M})$ such that $\mu = Q(a) = \sum_{x,y \in \widetilde{M}} a_{(x,y)} m_{xy}$ and $\norm{\mu}=\norm{a}_1$, we refer to $\mu$ as a \emph{convex series of molecules}  (see \cite{APS24,RRZ}). Notably, every finitely supported element in $\F(M)$ is a convex series (in fact, a convex sum) of molecules.  
	
	We briefly recall the so-called multiplication operators and their preadjoints, the weighted Lipschitz operators. Following \cite[Lemma~2.3]{APPP}, let $h : M \to \R$ be a Lipschitz function with bounded support, and let $K \subset M$ contain both the base point and the support of $h$. For $f \in \Lip_0(K)$, define
	\begin{equation}
		\label{eq:T_h}
		\mathsf{M}_h(f)(x)=\begin{cases}
			f(x)h(x) & \text{if } x\in K, \\
			0 & \text{if } x\notin K.
		\end{cases}
	\end{equation}
	Then $\mathsf{M}_h$ is a weak$^*$-to-weak$^*$ continuous operator from $\Lip_0(K)$ into $\Lip_0(M)$, called a \emph{multiplication operator}. Its preadjoint $T_h : \F(M) \to \F(K)$, given by $T_h(\delta(x)) = h(x)\,\delta(f(x))$, is referred to as a \emph{weighted Lipschitz operator}. Precise conditions on $h$ and $f$ ensuring boundedness of $\mathsf{M}_h$ and $T_h$ are characterized in \cite{ACP24}.

	\subsection{Some properties related to weak compactness} \label{section:Defs}
	
	A sequence $(x_n)_n$ in a Banach space $X$ is \textit{weakly Cauchy} if for every $x^* \in X^*$, the sequence $(\duality{x^*,x_n})_n$ converges.
	A subset $W$ of a Banach space $X$ is called \textit{weakly precompact} if every sequence $(x_n)_n \subset W$ admits a weakly Cauchy subsequence.
	Equivalently, by virtue of Rosenthal's $\ell_1$-theorem, $W$ is weakly precompact if it is bounded and no sequence in $W$ is equivalent to the unit vector basis of $\ell_1$. Now a Banach space:
	\begin{itemize}[leftmargin=*]
		\item has the \textit{Schur property} if every weakly null sequence is actually norm-null;
		\item is \textit{weakly sequentially complete} (wsc) if every weakly Cauchy sequence is weakly convergent.
	\end{itemize}
	Of course, every Schur space is weakly sequentially complete. It is also noteworthy that if $X$ has the Schur property then weak and norm compactness coincide. Also, if $X$ is weakly sequentially complete then weak precompactness and relative weak compactness coincide. 
	
	Next, a series $\sum x_n$ in a Banach space is \textit{weakly unconditionally Cauchy} (WUC) if $\sum_{n=1}^{\infty} |x^*(x_n)| < \infty$ for every $x^* \in X^*$. Now a subset $W \subset X$ is a \emph{V*-set} if for every WUC series $\sum x_n^*$ in $X^*$ one has
	\[
	\lim_n \sup_{w\in W} \abs{\duality{x_n^*,w}}=0.
	\]
	Observe that every weakly precompact set is a V*-set, but not conversely.  In fact, a Banach space
	\begin{itemize}[leftmargin=*]
		\item has \textit{property (V*)} if all V*-sets are relatively weakly compact (in that case, they agree with the weakly precompact sets too).
	\end{itemize}
	In particular, note that if $X$ has property (V*), then $X$ is weakly sequentially complete. The properties we just mentioned are summarized in the following diagram in terms of implications for closed sets:
	
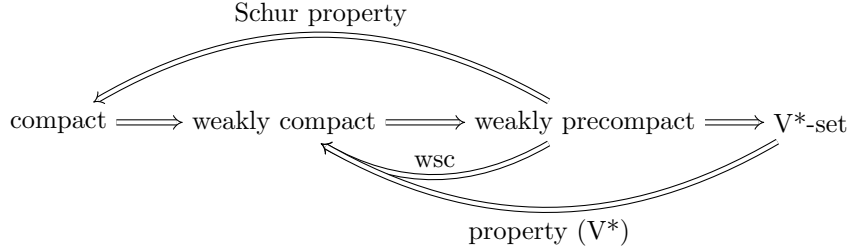
\begin{center}
		\begin{tikzpicture}
			\node (C) at (0,0) {compact};
			\node (WC) at (3,0) {weakly compact};
			\node (WPC) at (7,0) {weakly precompact};
			\node (V) at (10,0) {V*-set};
			\draw[-implies,double equal sign distance] (C) -- (WC);
			\draw[-implies,double equal sign distance] (WC) -- (WPC);		
			\draw[-implies,double equal sign distance] (WPC) -- (V);
			\draw[-implies,double equal sign distance] (WPC) to [bend left] node [above]{wsc} (WC);
			\draw[-implies,double equal sign distance] (V) to [bend left] node [midway,below]{property (V*)} (WC);
			\draw[-implies,double equal sign distance] (WPC) to [bend right] node [midway,above]{Schur property} (C);
		\end{tikzpicture}
	\end{center}

	\subsection{A brief overview in the framework of Lipschitz free spaces}
	
	Let us start with a remark with respect to norm compact sets in Lipschitz free spaces. It is a folklore fact that a set $K$ is relatively compact in a Banach space $X$ if and only if it is bounded and can be approximated by finite dimensional subspaces in the following sense: for every $\ep > 0$, there exists a finite dimensional subspace $F \subset X$ such that  $K \subset F + \ep B_X$. In terms of Lipschitz free spaces, we obtain the next similar statement:
	
	\begin{proposition} \label{norm-compactness}
		Let $M$ be any metric space and let $K \subset \F(M)$ be bounded. Then $K$ is relatively compact if and only if for every $\ep > 0$, there exists a finite subset $F \subset M$ such that $K \subset \F(F) + \ep B_{\F(M)}$.   
	\end{proposition}
	
	This proposition should be compared to \cite[Proposition~3.6]{ANPP} which implies that if $W$ is relatively weakly compact in $\F(M)$, then $W$ has \textit{Kalton's property}. That is, for every $\ep >0$ and every $\delta >0$, there exists a finite subset $F \subset M$  such that
	$$W \subset \F([F]_{\delta}) + \ep B_{\F(M)}.$$ 
	We emphasize that it is proved in \cite{ANPP} that $W \subset \F(M)$ has Kalton's property if and only if $W$ is tight, meaning that 
	$$ \forall \ep >0, \; \exists K \subset M \text{ compact,} \quad W \subset \F(K) + \ep B_{\F(M)}. $$
	It is also shown in the same paper that, for a complete metric space $M$, every weakly precompact set $W \subset \F(M)$ is tight. This last result has recently been slightly improved in \cite{APQ} since the authors prove that V*-sets are also tight.
	As we already mentioned above, tightness is quite far from characterizing relative weak precompactness in free spaces. Indeed, it is enough to consider any infinite and compact metric space $M$ while setting $W = B_{\F(M)}$.
	
	We conclude the section with some examples of complete metric spaces having at least one property which we introduced in the previous subsection. It is known \cite{AGPP21} that $\F(M)$ has the Schur property if and only if $M$ is purely 1-unrectifiable, which means that $M$ does not contain any bi-Lipschitz image of a subset $S$ of $\R$ with positive Lebesgue measure. Classical examples of such metric spaces are the middle-third Cantor set, the Koch curve, or even any $p$-snowflake metric space $(M,d^p)$ where $p \in (0,1)$. 
	Next, it is proved in \cite{CDW16} that $\F(M)$ is weakly sequentially complete for any $M \subset \R^n$. 
	This result is improved in \cite{KP18} since the authors show that $\F(M)$ has property (V*) for any compact subset $M$ of a superreflexive Banach space. Thanks to the compact reduction principle in  \cite{ANPP}, $\F(X)$ is weakly sequentially complete for any superreflexive Banach space.  Finally, the recent paper \cite{APQ} contains many interesting results concerning property (V*) in Lipschitz free spaces. Notably, it is proved that $\F(M)$ has property (V*) when the complete metric space $M$ is locally compact and purely 1-unrectifiable, an $\ell_p$ space for $1<p<\infty$, or a Carnot group.

	\section{Basic properties of uniformly regular sets}
	
	For convenience of the reader, we recall the following definitions from the introduction.
	
	\begin{definition}
		A bounded set $W \subset \F(M)$ is:
		\begin{itemize}
			\item \textit{uniformly regular} if for every $\ep >0$ and every open set $U \subset M$, there exists a compact subset $K \subset U$ such that $W \subset \F(K \cup U^c) + \ep B_{\F(M)}$.
			\item \textit{tight} if for every $\ep >0$, there exists a compact subset $K \subset M$ such that $W \subset \F(K) + \ep B_{\F(M)}$.
		\end{itemize}
	\end{definition}
	
	It is clear from the definition that uniform regularity implies tightness (simply take $U = M$ above), but the converse does not hold in general. Indeed, notice that if $M$ is compact, then any bounded subset of $\F(M)$ is tight but not necessarily uniformly regular:
	
	\begin{remark}
		\label{ex:ball_not_UR}
		If $M$ has at least one accumulation point $z$, then $W:=B_{\F(M)}$ is not uniformly regular. To see this, set $U:=M \setminus \set{z}$ and $\ep=\frac12$.
		Let $K \subset U$ be compact.
		Since $z$ is an accumulation point, we may find $x \in M\setminus (K\cup \set{z})$ such that $\frac{d(x,z)}{d(x,K)}\leq 1$. 
		Consider a 1-Lipschitz map $f$ such that $f \restricted_{K \cup \{z\}} = 0$ and $\duality{f,m_{xz}}=1$.
		Then, for every $\gamma \in \F(K \cup \{z\})$
		$$ \| m_{xz} - \gamma \|_{\F(M)} \geq \< f , m_{xz} - \gamma  \> =  \<f , m_{xz} \> = 1 > \frac{1}{2}.$$
		This proves that $m_{xz} \not\in  \F(K \cup \{z\}) + \frac{1}{2} B_{\F(M)}$.
		Hence $W$ is not uniformly regular as $m_{xz} \in W$.
		
		In contrast, if $M$ has no accumulation point, then compact subsets of $M$ are actually finite and therefore tightness, uniform regularity and (weak) compactness are all equivalent in this case by virtue of Proposition~\ref{norm-compactness}.
	\end{remark}
	
	We now continue with a few basic observations concerning uniformly regular sets, which will be useful in the sequel. These statements can be checked directly from the definition.

	\begin{proposition}\label{p:URbasic}
		\hfill
		\begin{enumerate}
			\item If $W \subset \F(M)$ is uniformly regular and $S \subset W$ then $S$ is uniformly regular.
			\item If $W_1$ and $W_2$ are uniformly regular in $\F(M)$, then $W_1+W_2$ is uniformly regular. 
			\item If $W \subset \F(M)$ is uniformly regular then the convex hull $\conv(W)$ is uniformly regular.
			\item ``Grothendieck's criterion'': A set $W \subset \F(M)$ is uniformly regular if and only if for every $\ep >0$, there exists a uniformly regular set $W_\ep \subset \F(M)$ such that $W \subset W_\ep + \ep B_{\F(M)}$.
		\end{enumerate}
	\end{proposition}

	\begin{remark} 
		Uniformly regular sets are not stable under isomorphisms (even isometric) between Lipschitz-free spaces. 
		Indeed, it is well known that if $C$ denotes the middle-third Cantor set and $M = \N$, then 
		\[
		\F(C) \equiv \ell_1 \equiv \F(M).
		\] 
		However, the forthcoming Proposition~\ref{prop:PerfectURnotwpreCPT} shows that there exists a uniformly regular set $W \subset \F(C)$ which is not weakly compact, whereas this is impossible in $\F(M)$, as noted in Remark~\ref{ex:ball_not_UR}.
	\end{remark}
	
	On the other hand, as we will show in the three lemmata below, uniformly regular sets are stable under the following special maps: linearizations of Lipschitz maps, weighted operators, and change of base-point maps. 
	
	\begin{lemma}\label{l:StabilityUnderFHat}
		Let $f \in \Lip_0(M,N)$ and assume that $W \subset \Free(M)$ is uniformly regular.
		Then $\hat{f}(W)$ is uniformly regular in $\F(N)$.
	\end{lemma}
	
	\begin{proof}
		Let $V$ be an open set in $N$ and $\ep>0$. Then $U:=f^{-1}(V)$ is an open set in $M$. Therefore, by assumption, there exists a compact set $K \subset U$ such that $W \subset \F(K \cup (M \setminus U)) + \ep B_{\F(M)}$. Since $f$ is continuous, $f(K)$ is a compact subset of $V$. 
		\[
		\begin{aligned}
			\hat{f}(W)&\subset \hat{f}\left(\Free(K\cup (M\setminus U))+\varepsilon B_{\Free(M)}\right)\\ 
			&\subset \hat{f}\left(\overline{\lspan}(K\cup (M\setminus U))\right)+\varepsilon \Lip(f) B_{\Free(N)}\\
			&\subset \overline{\lspan}(f(K \cup (M\setminus U)))+\varepsilon \Lip(f)B_{\Free(N)}\\
			&= \Free(f(K\cup (M\setminus U)))+\varepsilon \Lip(f)B_{\Free(N)}\\
			&\subset \Free(f(K) \cup (N\setminus V))+\varepsilon \Lip(f) B_{\Free(N)}.
		\end{aligned}
		\]
	\end{proof}

	\begin{lemma}\label{l:URthroughSupportOps}
		Let $N \subset M$ be a closed subset. Let $T:\Free(M) \to \Free(N)$ be a bounded linear operator such that for every $\mu \in \Free(M)$ we have $\supp(T\mu)\subset \supp(\mu)\cap N$.
		If $W \subset \Free(M)$ is uniformly regular then $T(W)$ is uniformly regular.
	\end{lemma}
	This lemma applies in particular to weighted operators (see the end of Subsection~\ref{section:LipFree}). 
	
	\begin{proof}
		Let $V$ be an open subset of $N$ and let $\varepsilon>0$. 
		There is an open $U \subset M$ such that $V=N\cap U$. 
		Notice that $N\setminus V=N\cap M\setminus U$.
		By hypothesis, there exists a compact $K \subset U$ such that
		$W\subset \Free(K\cup (M\setminus U))+\varepsilon B_{\Free(M)}$.
		Thus 
		\begin{align*}
			T(W) &\subset  \Free(\big(K\cup (M\setminus U) \big) \cap N)+\varepsilon \|T\| B_{\Free(N)}  \\
			&\subset \Free((K\cap N) \cup (N\setminus V))+\varepsilon\norm{T} B_{\Free(N)},
		\end{align*}
		which yields the conclusion.
	\end{proof}
	
	\begin{lemma}\label{l:StabilityUnderChangeOfBasepoint}
		Let $0\neq a \in M$ and let $\Free_0(M)$, resp. $\Free_a(M)$, be the canonical preduals of $\Lip_0(M)$, resp. $\Lip_a(M)$.
		Let $T:\Free_0(M)\to \Free_a(M)$ be the canonical isometry such that $T(\delta(x))=\delta(x)-\delta(0)$.
		Then $W \subset \Free_0(M)$ is uniformly regular iff $T(W) \subset \Free_a(M)$ is uniformly regular.
	\end{lemma}
	
	\begin{proof}
		Let $U \subset M$ be open set, let $\varepsilon>0$. 
		It is enough to show that there exists $K \subset U$ compact such that $T(W) \subset \Free_a(K \cup M\setminus U) +\varepsilon B_{\Free_a(M)}$. 
		First, notice that since $U$ is uniformly regular, there exists a compact $L \subset U$ such that for every $w\in W$ there are $x_i,y_i \in L \cup \set{0}\cup M\setminus U$ and $a_i \in \Real$, for $i\in \set{1,\ldots, n}$,  such that $\norm{w-\sum_{i=1}^n a_im_{x_iy_i}}\leq \varepsilon$.
		We now have $$\norm{Tw-\sum_{i=1}^n a_im_{x_iy_i}}=\norm{T\left(w-\sum_{i=1}^n a_im_{x_iy_i}\right)}=\norm{w-\sum_{i=1}^n a_im_{x_iy_i}}\leq \varepsilon.$$
		We can thus conclude by setting $K=L \cup \set{0}$ if $0\in U$ or $K=L$ if $0 \in M\setminus U$.
	\end{proof}

	In measure spaces, uniform regularity can be restated equivalently in terms of sequences of pairwise disjoint open sets (see e.g. \cite[Theorem 5.3.2]{AlbiacKalton}). As the main result in this section, we will now prove two similar equivalences for Lipschitz free spaces.
	
	\begin{theorem} \label{thmURequiv}
		Let $M$ be a complete metric space. For a bounded $W \subset \F(M)$, the following are equivalent:
		\begin{enumerate}
			\item[$(i)$] $W$ is uniformly regular,
			\item[$(ii)$] for every sequence $(U_n)$ of pairwise disjoint open sets in $M$
			$$
			\lim_{n\to\infty} \sup_{\gamma\in W} \dist(\gamma,\F(U_n^c)) = 0,
			$$
			\item[$(iii)$] for every sequence $(f_n)$ in $B_{\Lip_0(M)}$ such that the sets $U_n=\set{x\in M:f_n(x)\neq 0}$ are pairwise disjoint
			$$
			\lim_{n\to\infty} \sup_{\gamma\in W} \abs{\duality{f_n,\gamma}} = 0.
			$$
		\end{enumerate}
	\end{theorem}
	
	\begin{proof}
		
		$(i) \implies (ii)$: We prove the contrapositive. Assume that $(ii)$ fails. Passing to a subsequence if necessary, we may assume that there exist $\ep>0$ and a sequence of pairwise disjoint open sets $(U_n)$ such that $\sup_{\gamma \in W} \dist(\gamma , \F(U_n^c)) > \ep$. For every $n$, let $\gamma_n \in W$ such that  $\dist(\gamma_n , \F(U_n^c)) > \ep$. Set $U := \bigcup_n U_n$. Now for any compact $K \subset U$, there exists $N \in \N$ such that $K \subset \bigcup_{n=1}^N U_n$. In particular $K \cup U^c \subset U_{N+1}^c$. We deduce that 
		$$ \dist(\gamma_{N+1} , \F(K \cup U^c)) \geq \dist(\gamma_{N+1} , \F(U_{N+1}^c)) > \ep,$$
		and hence $W$ is not uniformly regular. 
		
		\medskip
		
		$(ii) \implies (iii)$: Let $\varepsilon>0$ and $(f_n)\subset B_{\Lip_0(M)}$ be such that the sets $U_n=\set{x\in M:f_n(x)\neq 0}$ are pairwise disjoint and that $\sup_{\gamma \in W} \abs{\duality{\gamma,f_n}}>\varepsilon$ for every $n$. 
		For a given $k$, let $\gamma \in W$ be such that $\abs{\duality{\gamma,f_k}}>\varepsilon$. 
		Let $\mu \in \Free(U_k^c)$. 
		We have $\norm{\gamma-\mu}\geq \abs{\duality{\gamma-\mu,f_k}}=\abs{\duality{\gamma,f_k}}>\varepsilon$. It follows that $\sup_{\gamma \in W}\dist(\gamma,\Free(U^c_k))>\varepsilon$ for every $k$.
		
		\medskip
		
		$(iii) \implies (ii)$: Assume that there exist pairwise disjoint open sets $U_n$ and elements $\gamma_n\in W$ such that $\dist(\gamma_n,\F(U_n^c))>\varepsilon$ for some $\varepsilon>0$. 
		For every $n$, the Hahn-Banach theorem yields $f_n\in B_{\Lip_0(M)}$ such that $\duality{f_n,\gamma_n}>\varepsilon$ and $f_n=0$ on $\F(U_n^c)$, i.e. $f_n=0$ on $U_n^c$. 
		Since $\duality{f_n,\gamma_n}>\varepsilon$, this contradicts $(iii)$.
		
		\medskip
		
		$(ii)$ and $(iii) \implies (i)$: This implication involves a much longer argument, so we will split the proof into several steps. Assume that $W$ satisfies $(ii)$ and $(iii)$.
		
		\medskip
		
		\textbf{Step 1}: \emph{Given $\varepsilon>0$, there is $r<\infty$ such that $W\subset \F(B(0,r))+\varepsilon B_{\F(M)}$.}
		
		We will prove the contrapositive, thus assume that there exists $\varepsilon>0$ such that for every $r>0$ we have $W\not \subset \F(B(0,r))+\varepsilon B_{\F(M)}$. Let $\gamma_1\in W$ be such that $\|\gamma_1\|>\varepsilon$. By the Hahn-Banach theorem there exists $f_1 \in B_{\Lip_0(M)}$ such that $\< f_1 , \gamma_1\> > \varepsilon$. We may additionally assume that there exists $r_1>0$ such that $f_1 \restricted_{B(0,r_1)^c} = 0$, as functions with bounded support are weak$^*$ dense in $B_{\Lip_0(M)}$ (see e.g. \cite[Lemma 2.2]{APPP}). Since $W\not \subset \F(B(0,r_1))+\varepsilon B_{\F(M)}$, there exists $\gamma_2\in W$ such that $\dist(\gamma_2 , \F(B(0,r_1)) > \varepsilon$. By the Hahn-Banach theorem there exists $f_2 \in B_{\Lip_0(M)}$ such that $\< f_2 , \gamma_2\> > \varepsilon$ and $f_2 \restricted_{B(0,r_1)} = 0$. As before, we may additionally assume that there exists $r_2>r_1$ such that $f_2 \restricted_{B(0,r_2)^c} = 0$. Continuing in such a way we will construct a sequence of functions $(f_n)\subset B_{\Lip_0(M)}$ such that the sets $U_n=\set{x\in M:f_n(x)\neq 0}$ are pairwise disjoint and $\< f_n , \gamma_n\> > \varepsilon$ for every $n\in \N$, in contradiction with $(iii)$.
		
		\medskip
		
		\textbf{Step 2}: \emph{$W$ is tight.}
		Again, we will prove the contrapositive, thus assume $W$ does not have the Kalton's property (which is equivalent to tightness \cite{ANPP}). Hence there exist $\varepsilon,\delta>0$ such that for every finite subset $F$ we have $W\not \subset \F([F]_\delta)+\varepsilon B_{\F(M)}$. By Step 1, there exists $r>\delta/4$ such that
		$$W\subset \F(B(0,r))+\frac{\varepsilon}{4} B_{\F(M)}.$$
		Let $\gamma_1\in W$ be such that $\|\gamma_1\|>\varepsilon$. Let $\mu_1\in \F(B(0,r))$ be such that $\|\gamma_1-\mu_1\|\le\frac{\varepsilon}{4}$. Then $\|\mu\|>\frac{3\varepsilon}{4}$ and thus there exists $g_1\in B_{\Lip_0(M)}$ such that $\< g_1 , \mu_1\> > \frac{3\varepsilon}{4}$. We may additionally assume that $g_1\restricted_{B(0,2r)^c}=0$ by replacing $g_1$ with a function $g'_1$ defined by $g'_1=g_1$ on $B(0,r)$, $g'_1=0$ on $B(0,2r)^c$, and extended to $M$ by McShane-Whitney theorem (see e.g. \cite[Theorem 1.33]{Weaver2}). Furthermore, $\< g_1 , \gamma_1\> > \frac{\varepsilon}{2}$. Pick a finitely supported element $\gamma_1'\in\F(M)$ such that $\< g_1 , \gamma_1'\> > \frac{\varepsilon}{2}$ and $\|\gamma_1-\gamma_1'\|<\frac{\varepsilon\delta}{16r}$. Let $F_1=\supp(\gamma_1')\cup\set{0}$ and define $f_1\colon [F_1]_{\delta/2}^c \cup F_1\rightarrow\mathbb{R}$ in such way that $f_1 \restricted_{[F_1]_{\delta/2}^c} = 0$ and $f_1 \restricted_{F_1} = \frac{\delta}{4r}g_1 \restricted_{F_1}$. Since $g_1\restricted_{B(0,2r)^c}=0$, for every $u\in F_1$ we have $\abs{g_1(u)}\le 2r$ and thus $\abs{f_1(u)}\le \frac{\delta}{2}$. Thus, since $\frac{\delta}{4r}\leq 1$, the Lipschitz constant of $f_1$ is at most 1. Extend $f_1$ to $M$ using McShane--Whitney Theorem. Then $f_1\in B_{\Lip_0(M)}$ and $\< f_1 , \gamma_1'\>=\frac{\delta}{4r}\< g_1 , \gamma_1'\> > \frac{\varepsilon\delta}{8r}$, and thus $\< f_1 , \gamma_1\>> \frac{\varepsilon\delta}{16r}$.
		
		We will continue recursively. Assume we have found $\gamma_1,\ldots,\gamma_n\in W$, finite sets $F_1,\ldots,F_n\subset M$, and functions $f_1,\ldots,f_n\in B_{\Lip_0(M)}$. Now, since  $W\not \subset \F\left(\left[\bigcup_{i=1}^nF_i\right]_\delta\right)+\varepsilon B_{\F(M)}$, there exists $\gamma_{n+1}\in W$ be such that
		$$\dist\left(\gamma_{n+1},\F\left(\left[\bigcup_{i=1}^nF_i\right]_\delta\right)\right)>\varepsilon.$$
		Let $\mu_{n+1}\in \F(B(0,r))$ be such that $\|\gamma_{n+1}-\mu_{n+1}\|\le \frac{\varepsilon}{4}$. Then
		$$\dist\left(\mu_{n+1},\F\left(\left[\bigcup_{i=1}^nF_i\right]_\delta\right)\right)>\frac{3\varepsilon}{4}.$$
		By the Hahn-Banach theorem, there exists $g_{n+1}\in B_{\Lip_0(M)}$ such that $\< g_{n+1} , \mu_{n+1}\> > \frac{3\varepsilon}{4}$ and $g_{n+1} \restricted_{\left[\bigcup_{i=1}^nF_i\right]_\delta}=0$. Since $\supp(\mu)\subset B(0,r)$, we may additionally assume that $g_{n+1} \restricted_{B(0,2r)}=0$ as above. Furthermore, $\< g_{n+1} , \gamma_{n+1}\> > \frac{\varepsilon}{2}$.  Pick a finitely supported element $\gamma_{n+1}'\in\F(M)$ such that $\< g_{n+1} , \gamma_{n+1}'\> > \frac{\varepsilon}{2}$ and $\|\gamma_{n+1}-\gamma_{n+1}'\|<\frac{\varepsilon\delta}{16r}$. Let
		$$F_{n+1}=\supp(\gamma_{n+1}')\setminus \left[\bigcup_{i=1}^nF_i\right]_\delta$$
		and define $f_{n+1}\colon [F_{n+1}]_{\delta/2}^c \cup F_{n+1}\cup\set{0}\rightarrow\mathbb{R}$ in such way that $f_{n+1}=0$ on $[F_{n+1}]_{\delta/2}^c$ and $f_{n+1} = \frac{\delta}{4r}g_{n+1}$ on $F_{n+1}\cup\set{0}$. For every $u\in F_{n+1}$ we have $\abs{g_{n+1}(u)}\le 2r$ and thus $\abs{f_{n+1}(u)}\le \frac{\delta}{2}$. Thus the Lipschitz constant of $f_{n+1}$ is at most 1. Extend $f_{n+1}$ to $M$ using McShane--Whitney Theorem. Then $f_{n+1}\in B_{\Lip_0(M)}$ and $\< f_{n+1} , \gamma_{n+1}'\>=\frac{\delta}{4r}\< g_{n+1} , \gamma_{n+1}'\> > \frac{\varepsilon\delta}{8r}$ and thus $\< f_{n+1} , \gamma_{n+1}\>> \frac{\varepsilon\delta}{16r}$ . 
		
		Now we have defined our sequences $(f_n)\subset B_{\Lip_0(M)}$ and $(\gamma_n)\subset W$ such that $\< f_n,\gamma_n \> > \frac{\varepsilon\delta}{16r}$. Let us check that the sets $U_n=\set{x\in M:f_n(x)\neq 0}$ are pairwise disjoint. Clearly $U_n\subset [F_n]_{\delta/2}$. Furthermore, since for every $i<n$ we have $ [F_i]_\delta\cap F_{n}=\emptyset$, then also $[F_i]_{\delta/2}\cap [F_{n}]_{\delta/2}=\emptyset$. Hence the sets $U_n$ are pairwise disjoint, and this contradicts $(iii)$. 
		
		\medskip
		
		\textbf{Step 3}: \emph{$W$ is uniformly regular.}
		Once again, we work by contradiction. So assume there exist $\varepsilon > 0$ and an open set $U \subset M$ such that for every compact $K \subset U$, $W \not\subset \F(K \cup U^c) + \varepsilon B_{\F(M)}$. Since $W$ is tight by Step 2, there exists a compact set $K\subseteq M$ such that $W\subseteq \F(K)+\frac{\varepsilon}{32}B_{\F(M)}$. 
		Let $K_0 = \set{0}\cap U$. By assumption there exists $\gamma_1 \in W$ such that $\dist(\gamma_1 , \F(K_0 \cup U^c)) > \varepsilon$.
		Pick $\mu_1\in\F(K)$ finitely supported such that $\|\gamma_1-\mu_1\|<\frac{\varepsilon}{16}$. Then $\dist(\mu_1 , \F(K_0 \cup U^c)) > \frac{\varepsilon}{2}$. 
		Let 
		\begin{align*}
			C_1&:=\supp(\mu_1)\setminus (K_0 \cup U^c),\\
			K_1&:=\bigcup_{u\in C_1}B\left(u,\frac{1}{2}d\left(u, K_0\cup U^c\right)\right)\cap K,\\
			U_1&:=\bigcup_{u\in C_1}B^O\left(u,\frac{1}{4}d(u,K_0\cup U^c)\right).
		\end{align*}
		Clearly $K_1$ is compact and $K_1,U_1\subseteq U\setminus K_0$.
		We claim that $\dist(\mu_1,\F(U_1^c))>\frac{\varepsilon}{8}$. Indeed, since $\dist(\mu_1 , \F(K_0 \cup U^c)) > \frac{\varepsilon}{2}$, by the Hahn-Banach theorem there exists $f_1 \in B_{\Lip_0(M)}$ such that $\< f_1 , \mu_1\> > \frac{\varepsilon}{2}$ and $f_1 \restricted_{K_0 \cup U^c} = 0$. Define $g_1\colon U_1^c \cup C_1\rightarrow\mathbb{R}$ in such way that $g_1 \restricted_{U_1^c} = 0$ and $g_1 \restricted_{C_1} = \frac{1}{4} f_1 \restricted_{C_1}$. For every $u\in C_1$ and $v\in U_1^c$ we have 
		$$|g_1(u)-g_1(v)|=\frac{1}{4}|f_1(u)|\le \frac{1}{4}d(u, K_0\cup U^c)\le d(u,v).$$
		Thus the Lipschitz constant of $g_1$ is at most 1. Extend $g_1$ to $M$ using McShane--Whitney theorem. Then $g_1\in B_{\Lip_0(M)}$ and $g_1 \restricted_{U_1^c} = 0$. Note that if $u\in \supp(\mu_1)\setminus C_1$, then $d(u,K_0 \cup U^c)=0$ and thus $g_1(u)=0=\frac{1}{4}f_1(u)$. Therefore $g_1 \restricted_{\supp(\mu_1)} = \frac{1}{4} f_1 \restricted_{\supp(\mu_1)}$, which gives us $\<g_1,\mu_1\>=\frac{1}{4}\<f_1,\mu_1\>>\frac{\varepsilon}{8}$. Hence $\dist(\mu_1,\F(U_1^c))>\frac{\varepsilon}{8}$, which implies $\dist(\gamma_1,\F(U_1^c))>\frac{\varepsilon}{16}$. 
		\medskip
		
		Now assume we have defined $\gamma_1,\ldots,\gamma_n\in W$, finite sets $C_1,\ldots,C_n\subseteq K$, compact sets $K_1,\ldots,K_n \subseteq U\cap K$, and pairwise distinct open sets $U_1,\ldots,U_n\subseteq U$ in such a way that: For every $i\in\set{1,\ldots,n}$ we have 
		\begin{align*}
			&\dist(\gamma_i,\F(U_i^c)) > \frac{\varepsilon}{16},\\
			&K_i=\bigcup_{u\in C_i}B\left(u,\frac{1}{2}d\left(u, \bigcup_{j=0}^{i-1}K_j\cup U^c\right)\right)\cap K, \text{ and }\\
			&U_i=\bigcup_{u\in C_i}B^O\left(u,\frac{1}{4}d\left(u, \bigcup_{j=0}^{i-1}K_j\cup U^c\right)\right).
		\end{align*}
		
		The set $\bigcup_{i=0}^{n}K_i\subseteq U$ is compact, thus by assumption there exists $\gamma_{n+1} \in W$ such that $\dist\left(\gamma_{n+1} , \F\left(\bigcup_{i=0}^{n}K_i \cup U^c\right)\right) > \varepsilon$. Pick $\mu_{n+1}\in\F(K)$ finitely supported such that $\|\gamma_{n+1}-\mu_{n+1}\|<\varepsilon/16$. Then $\dist(\mu_{n+1} , \F(\bigcup_{i=0}^{n}K_i \cup U^c)) > \varepsilon/2$. 
		Let 
		\begin{align*}
			C_{n+1}&:=\supp(\mu_{n+1}) \setminus \left(\bigcup_{i=0}^{n}K_i \cup U^c\right),\\
			K_{n+1}&:=\bigcup_{u\in C_{n+1}}B\left(u,\frac{1}{2}d\left(u,\bigcup_{i=0}^{n}K_i\cup U^c\right)\right)\cap K,\\
			U_{n+1}&:=\bigcup_{u\in C_{n+1}}B^O\left(u,\frac{1}{4}d\left(u,\bigcup_{i=0}^{n}K_i\cup U^c\right)\right).
		\end{align*}
		Clearly $K_{n+1}$ is compact and both $K_{n+1}$ and $U_{n+1}$ are subsets of $U\setminus \bigcup_{i=0}^{n}K_i$. Similarly as before, one can show that $\dist(\gamma_{n+1},\F(U_{n+1}^c))>\frac{\varepsilon}{16}$. 
		\medskip
		
		So, the only thing left to prove is $U_i\cap U_{n+1}=\emptyset$ for every $i\in\set{1,\ldots,n}$. Assume that there exists $v\in U_i\cap U_{n+1}$ for some $i\in\set{1,\ldots,n}$. Then there exists $u_1\in C_i$ and $u_2\in C_{n+1}$ such that $d(u_1,v)<\frac{1}{4}d(u_1,\bigcup_{j=0}^{i-1} K_j\cup U^c)$ and $d(u_2,v)<\frac{1}{4}d(u_2,\bigcup_{j=0}^{n}K_j\cup U^c)$. Clearly $u_1\in K_i$, and thus 
		$d(u_2,\bigcup_{j=0}^{n}K_j\cup U^c)\le d(u_1,u_2).$
		Furthermore, $u_2\in K$ and $u_2\notin K_{i}$, thus $d(u_1,u_2)\ge \frac{1}{2}d(u_1,\bigcup_{j=0}^{i-1} K_j\cup U^c)$ by the definition of $K_i$.
		Therefore
		\begin{align*}
			d(u_1,u_2)&\le d(u_1,v)+d(u_2,v)\\
			&<\frac{1}{4}\left(d\left(u_1, \bigcup_{j=0}^{i-1} K_j\cup U^c\right)+d\left(u_2,\bigcup_{j=0}^{n}K_j\cup U^c\right)\right)
			\leq d(u_1,u_2).
		\end{align*}
		This is a clear contradiction. Now we have found a sequence of pairwise disjoint open sets $(U_n)$ such that
		$$ \sup_{\gamma \in W} \dist(\gamma , \F(U_n^c)) \geq \dist(\gamma_n , \F(U_n^c)) > \frac{\varepsilon}{16} ,$$
		which contradicts $(ii)$. This ends the proof.
	\end{proof}
	
	We conclude this section by showing that the notion of uniform regularity does not depend on the ambient space. 
	Indeed, if $0 \in N \subset M$ and $W \subset \F(N)$ is uniformly regular relative to $\F(N)$, then it is clearly uniformly regular relative to $\F(M)$ as well. 
	The converse direction, namely whether $W$ remains uniformly regular relative to $\F(N)$ when it is uniformly regular relative to $\F(M)$, is less immediate. To establish this, we will use the following lemma.

	\begin{lemma}\label{l:DisjointExtension}
		Let $0\in N\subset M$ be closed. Suppose that $(f_i)_{i\in I}$ are elements of $B_{\Lip_0(N)}$ such that the sets $U_i=\set{x\in N:f_i(x)\neq 0}$ are pairwise disjoint. Then there exist extensions $\tilde{f}_i\in 2B_{\Lip_0(M)}$ such that $\tilde{f}_i\restricted_N=f_i$ and the sets $\set{x\in M:\tilde{f}_i(x)\neq 0}$ are pairwise disjoint.
	\end{lemma}
	
	\begin{proof}
		We may assume that all sets $U_i$ are non-empty, otherwise $f_i=0$ and we may take $\tilde{f}_i=0$. Since they are pairwise disjoint, we may also assume that $N\setminus U_i$ are non-empty.
		Define
		$$
		\tilde{U}_i = \bigcup_{x \in U_i} B^O\pare{x,\frac{d(x,N\setminus U_i)}{2}} ,
		$$
		where the balls are taken in $M$.
		Clearly the sets $(\tilde{U}_i)_{i\in I}$ are open, and moreover they are pairwise disjoint. Indeed, suppose $p\in\tilde{U}_i\cap\tilde{U}_j$ for $i\neq j\in I$. Then there exist $x\in U_i$, $y\in U_j$ such that
		$$
		p \in B^O\pare{x,\frac{d(x,N\setminus U_i)}2} \cap B^O\pare{y,\frac{d(y,N\setminus U_j)}2}
		$$
		therefore
		$$
		d(x,y) \leq d(x,p)+d(p,y) < \frac{d(x,N\setminus U_i)}2 + \frac{d(y,N\setminus U_j)}2 \leq \frac{d(x,y)}2 + \frac{d(y,x)}2 ,
		$$
		a contradiction.
		
		We define $\tilde{f}_i$ on $U_i\cup M\setminus \tilde{U}_i$ as
		\[
		\tilde{f}_i(x):=\begin{cases}
			f_i(x)& \mbox{ if } x\in U_i\\
			0& \mbox{ if } x \in M\setminus \tilde{U}_i.
		\end{cases}
		\]
		Notice that $N\setminus U_i\subset M\setminus \tilde{U}_i$ and $\tilde{f}_i\restriction_N=f_i$.
		We now check that $\|\tilde{f}_i\|_L\leq 2$.
		Let $x,y \in U_i \cup M\setminus \tilde{U}_i$.
		It is easy to see that we only need to check the situation when $x\in U_i$ and $y\in M\setminus (N\cup \tilde{U}_i)$.
		Fix $\varepsilon>0$ and $z\in N\setminus U_i$ such that $d(x,z)\leq (1+\varepsilon)d(x,N\setminus U_i)$.
		Then we have:
		\[
		\abs{\tilde{f}_i(x)-\tilde{f}_i(y)}=\abs{f_i(x)-f_i(z)}\leq d(x,z)
		\leq (1+\varepsilon) d(x,N\setminus U_i)\leq 2(1+\varepsilon) d(x,y).
		\]
		Letting $\varepsilon\to 0$ yields our claim.
		Now extend $\tilde{f}_i$ to $M$ using the McShane--Whitney theorem. This ends the proof, as $\tilde{f}_i$ vanishes outside of $\tilde{U}_i$.
	\end{proof}
	
	\begin{proposition}\label{p:URStableFreeSubspaces}
		Let $M$ be complete and $N \subset M$ be closed. Then a set $W \subset \Free(N)$ is uniformly regular with respect to $\Free(M)$ if and only if it is uniformly regular with respect to $\Free(N)$.
	\end{proposition}
	
	\begin{proof}
		The backward implication is a particular case of Lemma~\ref{l:StabilityUnderFHat}. 
		For the converse, suppose that $W$ is uniformly regular with $\Free(M)$ as the ambient space. We may assume $0\in N$. Let $(f_n)$ be a sequence in $B_{\Lip_0(N)}$ such that the sets $\set{x\in M:f_n(x)\neq 0}$ are pairwise disjoint. Use Lemma \ref{l:DisjointExtension} to find a sequence of extensions $(\tilde{f}_n)$ in $2B_{\Lip_0(M)}$ with the same disjointness property. By hypothesis and Theorem \ref{thmURequiv}~(iii), $\<\tilde{f}_n,\gamma\>$ converges to $0$ uniformly on $\gamma\in W$. Since $W\subset\lipfree{N}$, we have that $\<\tilde{f}_n,\gamma\>=\<\tilde{f}_n\restricted_N,\gamma\>=\<f_n,\gamma\>$ also converges to $0$ uniformly. Thus, another application of Theorem \ref{thmURequiv}~(iii) shows that $W$ is uniformly regular in $\Free(N)$.
	\end{proof}

	%----------------------------------------------------------
	\section{Uniform regularity versus weak compactness} \label{SectionURnotWC}

	To start, we derive Theorem~\ref{thmA} as an immediate corollary of Theorem~\ref{thmURequiv}.
	
	\begin{corollary} \label{thmVstarUR}
		Let $M$ be a complete metric space. If $W\subset\lipfree{M}$ is a V*-set, in particular if it is weakly compact, then it is uniformly regular.
	\end{corollary}
	
	\begin{proof}
		We show that $W$ satisfies condition $(iii)$ in Theorem \ref{thmURequiv}. Let $(f_n)$ be a sequence in $B_{\Lip_0(M)}$ such that the sets $U_n=\set{x\in M:f_n(x)\neq 0}$ are pairwise disjoint. Then $\sum_n f_n$ is a WUC series in $\Lip_0(M)$ by \cite[Lemma 6]{AP_normality} (see also the proof of \cite[Lemma 1.5]{CCGMR}). Since $W$ is a V*-set, we have $\lim_n\sup_{\gamma\in W}\abs{\duality{f_n,\gamma}}=0$ as required.
	\end{proof}
	
	This argument highlights nicely the difference between uniformly regular sets and V*-sets: $W$ is a V*-set if $\lim_{n}\sup_{\gamma\in W}\abs{\duality{f_k,\gamma}}=0$ for all WUC series $\sum_n f_n$ in $\Lip_0(M)$, while $W$ is uniformly regular if the same holds for WUC series which moreover satisfy that the sets $\set{f_n\neq 0}$ are pairwise disjoint.
	
	This points to the fact that V*-sets and uniformly regular sets may not coincide in general. Although it is not difficult to find a set $W \subset \F(M)$ which is tight but not V* (or not weakly compact), it is more challenging to identify a counterexample in the case of uniform regularity. Our next aim is to address this challenge. To begin, we present a technical lemma.
	
	\begin{lemma}\label{lem_distance_to_element}
		Let $M$ be a metric space, and let $U\subseteq M$ be open. If $\varepsilon\in(0,1)$ and $\gamma=\frac{1}{n}\sum_{i=1}^nm_{u_iv_i}\in\lipfree{M}$ are such that 
		$\dist(\gamma, \lipfree{U^c})>\varepsilon$,
		then 
		\[\big|\big\{i\in\{1,\ldots,n\}: B(u_i,r_i)\subseteq U \text{ or } B(v_i,r_i)\subseteq U\big\}\big| > \frac{n\varepsilon}{2},\]
		where $r_i=\frac{\varepsilon}{8}d(u_i,v_i)$, $i\in\{1,\ldots,n\}$.
	\end{lemma}
	
	\begin{proof}
		Denote 
		\[I=\big\{i\in\{1,\ldots,n\}: B(u_i,r_i)\subseteq U \text{ or } B(v_i,r_i)\subseteq U\big\}.\]
		Note that for every $i\in \{1,\ldots,n\}\setminus I$, there exist $x_i\in B(u_i,r_i)\cap U^c$ and $y_i\in B(v_i,r_i)\cap U^c$. Thus, by \cite[Lemma~1.2]{VeeorgStudia}: 
		\begin{align*}
			\|m_{u_iv_i}-m_{x_iy_i}\|&=\frac{d(u_i,x_i)+d(v_i,y_i)+|d(u_i,v_i)-d(x_i,y_i)|}{\max\{d(u_i,v_i),d(x_i,y_i)\}}\\
			&\le \frac{d(u_i,x_i)+d(v_i,y_i)+d(u_i,x_i)+d(v_i,y_i)}{d(u_i,v_i)}\\
			&\le \frac{4r_i}{d(u_i,v_i)}=\frac{\varepsilon}{2}.
		\end{align*}
		Let $\nu=\frac{1}{n}\sum_{i\notin I}m_{x_iy_i}$. Clearly $\nu\in \lipfree{U^c}$ and 
		$$
		\varepsilon < \|\gamma-\nu\|\le\frac{1}{n}\sum_{i\in I}\|m_{u_iv_i}\|+\frac{1}{n}\sum_{i\notin I}\|m_{u_iv_i}-m_{x_iy_i}\| \le \frac{|I|}{n}+\frac{\varepsilon}{2}
		$$
		therefore $|I|>n\varepsilon/2$.
	\end{proof}
	
	\begin{proposition} \label{prop:PerfectURnotwpreCPT}
		Let $M$ be a complete metric space containing a non-empty perfect subset. Then $\Free(M)$ contains a subset that is uniformly regular but not a V*-set. 
	\end{proposition}
	\begin{proof}
		For the following construction, we use the notation $\set{0,1}^n$ for the set of finite sequences of $0$'s and $1$'s of length $n\in\N$, and $\set{0,1}^\N$ for the set of infinite sequences of $0$'s and $1$'s. If $s\in\set{0,1}^n$ and $i\in\set{0,1}$, then $s^\smallfrown i$ denotes the sequence in $\set{0,1}^{n+1}$ composed of $s$ followed by the digit $i$.
		
		Without loss of generality, we may assume that 
		$M$ is a complete perfect metric space.
		There exists a non-empty closed perfect subset $C \subset M$ which can be written as $C=\bigcap_n \bigcup S_n$
		where $S_n=\set{A_s:s \in \set{0,1}^n}$ and moreover we have for all  $s \in \set{0,1}^{<\N}$ and all $i \in \set{0,1}$,
		\begin{enumerate}
			\item $A_s$ is perfect, closed, and non-empty,
			\item $A_{s^\smallfrown i} \subset A_s$,
			\item $\diam(A_{s^\smallfrown i})\leq \frac13\diam(A_s)$,
			\item $d(A_{s^\smallfrown 0},A_{s^\smallfrown 1})\geq \frac13 \diam(A_s)$.
		\end{enumerate}
		Such a system indeed exists. To see this, recall that the closure of every open subset of a perfect set is perfect. 
		Set $A_\emptyset:=\cl{B^O(0,\frac12)}$ and when $A_s$ has been defined take $x_0,x_1 \in A_s$ such that $d(x_0,x_1)>\frac79 \diam(A_s)$, then define $A_{s^\smallfrown i}:=\overline{B^O(x_i,\frac19 \diam(A_s))\cap A_s}$.
		The required properties follow easily.
		
		Given $C$ and $\set{S_n}_{n=0}^\infty$ as above, we assume without loss of generality that $\diam(A_\emptyset)=1$. Now we define a system $A'_s$, $s\in\set{0,1}^{<\N}$ of non-degenerate closed subintervals of $[0,1]$ satisfying properties (1)-(4) above and such that $\diam(A'_s)=\diam(A_s)$.
		First, let $A_\emptyset'=[0,1]$.
		Next, assume that $A_s'$ has been defined and choose $A_{s^\smallfrown 0}'$ and $A_{s^\smallfrown 1}'$ to be placed on the left and right sides of $A'_s$, respectively, i.e.
		\begin{align*}
			A'_{s^\smallfrown 0} &:= [\min A_s',\min A_s'+\diam(A_{s^\smallfrown 0})] \\
			A'_{s^\smallfrown 1} &:= [\max A'_s -\diam(A_{s^\smallfrown 1}),\max A'_s].
		\end{align*}
		It is clear that they satisfy the required properties as well.
		For further reference, let us denote $u_{s^\smallfrown 0}':=\max A_{s^\smallfrown 0}'$ and $u_{s^\smallfrown 1}':=\min A_{s^\smallfrown 1}'$, and notice that the interval $(u'_{s^\smallfrown 0},u'_{s^\smallfrown 1})$ is contained in $A_s'$ but intersects neither $A'_{s^\smallfrown 0}$ nor $A'_{s^\smallfrown 1}$.
		
		We define the generalized Cantor set $C':=\bigcap_n \bigcup S_n'$ where $S'_n=\set{A_s':s\in \set{0,1}^n}$.
		To simplify notation, for $s=(s_k)\in \set{0,1}^{\N}$ we write $s_{\restricted n} := (s_1,\ldots , s_n)$.
		It is well known, and easily verified, that the above properties imply that for every 
		$s=(s_n)\in \{0,1\}^{\mathbb{N}}$ there exist points $x_s$ and $x_s'$ such that 
		\[
		\{x_s\} = \bigcap_{n} A_{s_{\restricted n}}, 
		\qquad 
		\{x_s'\} = \bigcap_{n} A_{s_{\restricted n}}',
		\]
		and moreover this completely describes all the points of $C$ and $C'$, respectively. 
		We claim that the map $\varphi : x_s \mapsto x_s'$ is bi-Lipschitz.
		Indeed, let $x_s \neq x_t \in C$ be given and let $n \in \N$ be the smallest integer such that $s_{n+1}\neq t_{n+1}$.
		Then $x_s', x_t' \in A_{s_{\restricted n}}'$ and therefore $\abs{x_s'-x_t'}\leq \diam(A'_{s_{\restricted n}})=\diam(A_{s_{\restricted n}})$.
		On the other hand, $x_s \in A_{s_{\restricted n}^\smallfrown i}$ and $x_t \in A_{s_{\restricted n}^\smallfrown j}$ with $i \neq j$, so that $d(x_s, x_t) \geq \frac{1}{3} \, \mathrm{diam}(A_{s_{\restricted n}})$.
		Therefore,
		\[
		|x_s' - x_t'| \leq 3 \, d(x_s, x_t).
		\]
		A Lipschitz estimate for the inverse is obtained by interchanging the roles of $C$ and~$C'$.
		Notice also that the points $u_s'$ defined above belong to the generalized Cantor set $C'$ and that the corresponding points $u_s=\varphi^{-1}(u_s')$ satisfy $u_{s} \in C\cap A_{s}$.
		
		\medskip
		
		Let now 
		$$\gamma_n := \frac{1}{2^{n-1}}\sum_{s \in \set{0,1}^{n-1}} m_{u_{s^\smallfrown 1} u_{s^\smallfrown 0}}\in\F(C) 
		\,\mbox{ and }\, 
		\gamma_n'=\hat{\varphi}(\gamma_n)\in\F(C').$$
		We denote $W:=\set{\gamma_n:n\in \N}$ and $W'=\hat{\varphi}(W)$.
		Assume towards a contradiction that $W$ is a V*-set. 
		Consider a Lipschitz extension $\psi:M\to [0,1]$ of $\varphi$.
		Then $W'=\hat{\varphi}(W)=\hat{\psi}(W)$ (see for example~\cite[Proposition~2.6]{APPP}). 
		Since V*-sets are stable under bounded linear operators, $W'$ is a V*-set in $\Free([0,1])=L_1$. 
		Since $L_1$ enjoys property (V*), $W$ must be relatively weakly compact. 
		But $(\gamma_n')_n$ is equivalent to the $\ell_1$-basis as we will show now.
		To do this, we use the usual isometry $T : \delta(x) \in \F([0,1]) \mapsto \indicator{[0,x]} \in L^1([0,1])$. Under this transformation, we have
		$$f_n := T(\gamma_n') = \frac{1}{2^{n-1}}\sum_{s \in \set{0,1}^{n-1}} \frac{\indicator{(u_{s^\smallfrown 0}' u_{s^\smallfrown 1}')}}{d(u_{s^\smallfrown 1},u_{s^\smallfrown 0})}.$$
		By the choice of points $(u'_s)$, it is evident that the maps $f_n$ have disjoint support. 
		Moreover, since $\varphi$ is bi-Lipschitz and since the sets $A_{s}'$ for $s\in \set{0,1}^{n-1}$ are pairwise disjoint, we have that  $\frac13\leq \norm{f_n} \leq 3$, indicating that $(f_n)_n$ is equivalent to the $\ell_1$-basis (see e.g. \cite[Lemma~5.1.1]{AlbiacKalton}).
		It follows that $(\gamma_n)$ is also equivalent  to the $\ell_1$-basis, as $\gamma_n'=\hat{\varphi}(\gamma_n)$ and $\hat{\varphi}$ is an isomorphism.
		Hence $W$ is not a V*-set.
		
		It remains to be shown that the set $W=\set{\gamma_n:n \in \N}$ is uniformly regular.
		Let $(U_n)$ be a sequence of pairwise disjoint open subsets of $M$.
		In order to reach contradiction via Theorem \ref{thmURequiv} we assume that \[\displaystyle\limsup_{k \to \infty} \sup_{n\in \N} \dist(\gamma_n,\Free(U^c_k))>\varepsilon\] for some $\varepsilon>0$.
		Let $l \in \N$ be such that $3^{-l}<\frac{\varepsilon}{8\cdot 3}$. 
		Since finite subsets of $\Free(M)$ are uniformly regular, we may choose sequences $(n_i)\subset \N$ and $(k_i)\subset \N$ such that $n_{i+1}>n_i+l$ and $k_{i+1}>k_i$ and such that $\dist(\gamma_{n_i},\Free(U^c_{k_i}))>\varepsilon$. 
		To simplify the notation we write, as we may, 
		$\dist(\gamma_{n_k},\Free(U^c_{k}))>\varepsilon$ for every $k\in \N$. 
		
		We now apply Lemma~\ref{lem_distance_to_element} (and the property (4) of the Cantor scheme above)  to obtain for every $k\in \N$ that 
		$$
		\cardinality{\set{s\in \set{0,1}^{n_k-1}:B(u_{s^\smallfrown 0},r_s) \subset U_k \mbox{ or }B(u_{s^\smallfrown 1},r_s) \subset U_k}}>2^{n_k-2}\cdot \varepsilon
		$$
		where $r_s=\frac\varepsilon8\frac{\diam(A_s)}3$.
		Notice that, since $3^{-l}<\frac\varepsilon{8\cdot 3}$, we have for every $s \in \set{0,1}^{n_k-1}$ and every $t\in \set{0,1}^l$ that 
		\[
		\diam(A_{s^\smallfrown t})\leq \left(\frac13\right)^l \diam(A_s)\leq \frac\varepsilon8\frac{\diam(A_s)}3=r_s.
		\]
		It follows that $A_{s^\smallfrown t} \subset U_k$ when 
		$B(u_{s^\smallfrown i},r_s) \subset U_k$ and $u_{s^\smallfrown i} \in A_{s^\smallfrown t}$.
		That is, for every $k$, at least $2^{n_k-2}\cdot \varepsilon$ different sets $A_s\in S_{n_k-1+l}$ are included in $U_k$.
		Further, because of the splitting-in-two-property of the Cantor scheme, for every $m \in \N$,  and every $k$, we will have at least $2^m\cdot 2^{n_k-2}\cdot \varepsilon$ different sets $A_s \in S_{n_k-1+l+m}$ that are included in $U_k$.
		Let $K \in \N$. 
		For every $k=1,\ldots,K$ we choose $m$ above as $n_K-n_k$.
		Thus at least $2^{n_K-n_k}\cdot 2^{n_k-2}\cdot\varepsilon=2^{n_K-2}\cdot \varepsilon$ different sets $A_s\in S_{n_k-1+l+n_K-n_k}=S_{n_K-1+l}$ are included in $U_k$. 
		Thus, since the sets $(U_k)$ are pairwise disjoint, $K\cdot 2^{n_K-2}\cdot \varepsilon$ different members of $S_{n_K-1+l}$ are included in $\bigcup_{k=1}^K U_k$.
		But the cardinality of $S_{n_K-1+l}$ is $2^{n_K-1+l}$.
		Therefore, if $K$ is chosen large enough, the fact that $(U_k)$ are pairwise disjoint leads to a contradiction.
	\end{proof}
	
	%As the byproduct of the above proof, we can state the following lemma of independent interest, which is probably well known but we could not locate in the literature.
	Note that the construction in the proof of Proposition \ref{prop:PerfectURnotwpreCPT} also yields the following fact of independent interest. We thank the anonymous referee for pointing out that a different proof thereof can be found e.g. in \cite[Theorem 17]{Geschke}.

	\begin{lemma}\label{lm:perfect}
		Let $M$ be a perfect complete metric space. 
		Then $M$ contains a bi-Lipschitz copy of a perfect subset of $\R$.
	\end{lemma}

	\begin{definition} \label{defURproperties}
		We will say that $M$ has property:
		\begin{itemize}[leftmargin=*]
			\item (UR$\subset$cpt) if every uniformly regular set $W \subset \Free(M)$ is relatively compact.
			\item (UR$\subset$w-cpt) if every uniformly regular set $W \subset \Free(M)$ is relatively weakly compact.
			\item (UR$\subset$w-precpt) if every uniformly regular set $W \subset \Free(M)$ is weakly precompact.
			\item (UR$\subset$V*) if every uniformly regular set $W \subset \Free(M)$ is a V*-set.
		\end{itemize}
	\end{definition}
	
	It is clear that (UR$\subset$cpt) $\implies$ (UR$\subset$w-cpt) $\implies$ (UR$\subset$w-precpt) $\implies $(UR$\subset$V*).
	We will show in Theorem \ref{c:URisCPTChar} that these four properties are, in fact, equivalent, and we will characterize precisely for which metric spaces $M$ they hold.

	Next, for a metric space $M$, we define its subset $$M_k=\set{x\in M:d(x,0)\in [2^{k},2^{k+2}]} \cup \{0\}$$ for every $k \in \Integer$. 
	The well-known Kalton decomposition \cite[Proposition~4.3]{Kalton04} asserts that there exist Lipschitz functions 
	$\varphi_n : M \to \mathbb{R}$ with $\supp(\varphi_n) \subset M_n$ such that the pre-adjoints 
	$T_n=T_{\varphi_n}$ of the associated multiplication operators $\mathsf{M}_{\varphi_n} : \Lip_0(M_n ) \to \Lip_0(M)$
	satisfy that 
	\[
	T := \bigoplus_{n \in \mathbb{Z}} T_n : \Free(M) \to X, 
	\qquad \text{where } X = \left(\sum_{n \in \mathbb{Z}} \Free(M_n )\right)_{\ell_1},
	\] 
	is a bounded operator with complemented range. More precisely, we have $S \circ T = \mathrm{Id}_{\Free(M)}$,
	where $S : \left(\sum_{n \in \mathbb{Z}} \Free(M_n )\right)_{\ell_1} \to \Free(M)$
	is the sum of the coordinates.  
	The following propositions will be our main workhorses.

	\begin{proposition}\label{p:KaltonDecomposition}
		Let $M$ be a metric space, let $W\subset\Free(M)$ be uniformly regular.
		Then for every $\varepsilon>0$ there exists a finite set $F \subset \Integer$ such that $$T(W) \subset T_F(W)+\varepsilon B_{X}$$
		where $T_F=\bigoplus_{n \in F} T_n$.
	\end{proposition}
	\begin{proof}
		Let $U=\bigcup_{n\in \Integer} M_n = M\setminus \{0\}$ and let $\varepsilon>0$ be fixed.
		There exists a compact $K \subset U$ such that $W \subset \Free(K \cup \set{0}) +\frac{\varepsilon}{2\norm{T}}B_{\Free(M)}$.
		Since $\inf_{x\in K}d(0,x)>0$ and $\sup_{x\in K} d(0,x)<\infty$, there is a finite set $F \subset \Integer$ such that $\varphi_n= 0$ on $K$ whenever $n \notin F$.
		Let $w \in W$ and $\tilde{w} \in \Free(K\cup\set{0})$ such that $\norm{w-\tilde{w}}\leq \frac{\varepsilon}{2\norm{T}}$.
		Then $$\norm{Tw-T_Fw}=\norm{Tw-T\tilde{w}+T_F\tilde{w}-T_Fw}\leq 2\norm{T}\norm{w-\tilde{w}}\leq \varepsilon$$ as desired.
	\end{proof}
	
	\begin{proposition} \label{Prop:KDUR}
		Let $M$ be a metric space such that $M_n$ has property (UR$\subset$V*) for every $n \in \Z$. Then $M$ has property (UR$\subset$V*). The same holds for properties (UR$\subset$w-precpt), (UR$\subset$w-cpt), and (UR$\subset$cpt).
	\end{proposition}
	\begin{proof} 
		We present the proof only for (UR$\subset$V*), since a similar argument applies to the other properties.
		Let $W \subset \Free(M)$ be a uniformly regular set. 
		By Proposition~\ref{p:KaltonDecomposition}, for every $\varepsilon>0$, there exists a finite $F\subset\Z$ such that the set $T(W)$ is $\varepsilon$-close to the set $T_F(W)$.
		We claim that $T_F(W)$ is a V*-set.
		Indeed, let $(f_k) \subset \left(\sum_{n\in F} \Free(M_n )\right)_{\ell_1}^*$ be such that $\sum f_k$ is a WUC series. For each $k \in \N$, let us write $f_k=(f_{k,n})_{n \in F}$ with $f_{k,n} \in \Lip_0(M_n)$. By Lemma~\ref{l:URthroughSupportOps}, each $T_n(W)$ is a uniformly regular set and thus a V*-set by assumption.
		Then, 
		$$\sup_{w\in W}\abs{\duality{T_Fw,f_k}}\leq \sum_{n\in F}\sup_{w\in W}\abs{\duality{T_nw,f_{k,n}}}\to 0$$ as $k\to\infty$.
		Thus, by the Grothendieck criterion for V*-sets (see e.g. \cite[Fact~1.6]{APQ}), $T(W)$ is a V*-set. Finally, since the image of V*-set by a bounded operator is a V*-set, $W=S\circ T(W)$ is a V*-set.
	\end{proof}
	
	In the next statement, if $\alpha$ is an ordinal then $M^{(\alpha)}$ denotes the Cantor-Bendixson derived set of order $\alpha$.
	
	\begin{proposition}
		\label{Prop:CountablePropertyUR}
		Let $M$ be a compact such that $M^{(\alpha)}$ is finite for some ordinal~$\alpha$. 
		Then $M$ has property (UR$\subset$cpt).
	\end{proposition} 
	
	\begin{proof}
		The proof proceeds by transfinite induction on $\alpha$. Let us denote:
		\smallskip
		
		\noindent ($H_\alpha$): For every compact metric space $M$, if $M^{(\alpha)}$ is finite, then $\F(M)$ has property (UR$\subset$cpt).
		\smallskip
		
		If $M$ is finite then $\dim\Free(M)<\infty$, and so every bounded set is relatively compact.
		\smallskip
		
		Let $\alpha \geq 1$ be an ordinal  such that $M^{(\alpha)}$ is finite and assume that $(H_\beta)$ is true for any $\beta < \alpha$. We will prove that $(H_\alpha)$ is true by induction on $k=|M^{(\alpha)}|\geq 0$. Suppose first that $k=0$, that is $M^{(\alpha)} = \emptyset$. By compactness of $M$, there exists $\beta < \alpha$ such that $M^{(\beta)} = \emptyset$, so $\F(M)$ has property (UR$\subset$cpt) by induction hypothesis ($H_\beta$).
		
		Now suppose that $M^{(\alpha)} = \{x_1,\ldots,x_k\}$ with $k\geq 1$. Thanks to Lemma~\ref{l:StabilityUnderChangeOfBasepoint}, we may assume without loss of generality that $x_1 = 0$. Then it follows that for every $n\in \Z$, $M_n^{(\alpha)}$ contains at most $k-1$ points. Therefore, by the second induction hypothesis, $\F(M_n)$ has property (UR$\subset$cpt) for every $n \in \Z$. Consequently, $\F(M)$ has property (UR$\subset$cpt) thanks to Proposition~\ref{Prop:KDUR}.
	\end{proof}

	\begin{theorem}\label{c:URisCPTChar}
		Let $M$ be a complete metric space. Then the following assertions are equivalent:
		\begin{enumerate}[$(i)$]
			\item $M$ is scattered.
			\item $M$ has property (UR$\subset$cpt).
			\item $M$ has property (UR$\subset$w-cpt).
			\item $M$ has property (UR$\subset$w-precpt).
			\item $M$ has property (UR$\subset$V*).
		\end{enumerate}
	\end{theorem}
	
	\begin{proof}
		First, observe that the chain of implications $(ii) \Rightarrow (iii) \Rightarrow (iv) \Rightarrow (v)$ is immediate. 
		On the other hand, if $M$ is not scattered, then it contains a non-empty perfect subset, and we can apply Proposition~\ref{prop:PerfectURnotwpreCPT} to obtain $(v) \Rightarrow (i)$. 
		Consequently, it remains only to prove $(i) \Rightarrow (ii)$. 
		Note that, under the additional assumption that $M$ is compact, scatteredness is equivalent to countability, and therefore this implication follows from Proposition~\ref{Prop:CountablePropertyUR}. 
		To establish it in the general case, we proceed by means of a compact reduction argument.
		Let $M$ be scattered.
		Assume that $W\subset \Free(M)$ is uniformly regular and fix $\varepsilon>0$.
		Then it is tight and~\cite[Theorem 2.3]{ANPP} implies that there exist a compact $K\subset M$ and a sequence of weighted operators $T_n:\Free(M)\to \Free(M)$ which converges uniformly on $W$ to a mapping $T:W \to \Free(K)$, 
		and moreover $W \subset T(W)+\varepsilon B_{\Free(M)}$.
		So, by Lemma~\ref{l:URthroughSupportOps}, the sets $T_n(W)$ are uniformly regular. 
		By the uniform convergence the set $T(W)$ is arbitrarily close to these sets.
		Proposition~\ref{p:URbasic}~(4) (``Grothendieck's criterion'') thus shows that $T(W)$ is uniformly regular with respect to $\Free(M)$.
		Proposition~\ref{p:URStableFreeSubspaces} further shows that $T(W)$ is uniformly regular with respect to $\Free(K)$.
		The compact case implies that $T(W)$ is relatively compact in $\Free(K)$, hence in $\Free(M)$. 
		Recalling that $W \subset T(W)+\varepsilon B_{\Free(M)}$, ``Grothendieck's criterion for compacts'' implies that $W$ is relatively compact itself. This shows $(i) \implies (ii)$.
	\end{proof}

	Corrollary~\ref{c:CorollaryC} is now obtained as a consequence of Corollaries \ref{thmVstarUR} and \ref{c:URisCPTChar}.
	This previously unknown result improves~\cite[Corollary 2.10]{APQ}.
	
	\begin{corollary}\label{cr:scattered_v*}
		If $M$ is complete and scattered, then $\F(M)$ has property (V*).
	\end{corollary}

	\bigskip
	
	To conclude this section, we introduce another class of subsets of Lipschitz free spaces for which uniform regularity serves as a characterization of relative compactness.
	If $k \in \N$, let $\mathcal{FS}_k(M)$ denote the subset of elements in $\F(M)$ whose support contains at most $k$ elements. For sequences in $\mathcal{FS}_k(M)$, norm and weak convergence are actually equivalent \cite[Theorem C]{ACP21}. 
	
	\begin{proposition}
		Let $M$ be a complete metric space and let $W \subset \mathcal{FS}_k(M)$ be bounded. Then $W$ is relatively compact if and only if it is uniformly regular.
	\end{proposition}
	
	\begin{proof}
		By Corollary~\ref{thmVstarUR}, it suffices to show that if 
		$ W \subset \mathcal{FS}_k(M) $ is uniformly regular, then $ W $ is relatively compact.  
		Let $(\gamma_n)_n$ be a sequence in $W$. By assumption, each element $\gamma_n$ can be written as
		\[
		\gamma_n = \sum_{i=1}^{k} a_i^n \, \delta(x_i^n), \qquad \text{for some $a_i^n\in\R$ and $x_i^n\in M$.}
		\]
		By passing to a subsequence (which we continue to denote $(\gamma_n)_n$ for simplicity), 
		we may assume that for each $i \in \{1,\ldots,k\}$, the sequence $(x_i^n)_n$ is either Cauchy (hence convergent) or uniformly discrete.
		Let
		\[
		S := \overline{\{x_i^n : n \in \mathbb{N}, \ 1 \le i \le k\}} \subset M.
		\]
		Then $S$ is a complete subset of $M$ with at most $k$ accumulation points, hence scattered. 
		Moreover, since every $\gamma_n$ is supported on $S$, we have 
		$\{\gamma_n : n \in \mathbb{N}\} \subset \mathcal{F}(S)$.
		By Proposition~\ref{p:URbasic}~(1) together with 
		Proposition~\ref{p:URStableFreeSubspaces}, the set 
		$\{\gamma_n : n \in \mathbb{N}\}$ is uniformly regular in $\mathcal{F}(S)$. 
		Applying Corollary~\ref{c:URisCPTChar}, we conclude that 
		$(\gamma_n)_n$ admits a convergent subsequence. 
		Hence $W$ is relatively compact, as desired.
	\end{proof}

	%----------------------------------------------------------
	\section{Some insights from the case of real trees} \label{section:Trees}
	
	An $\R$-tree is an arc-connected metric space $(T, d)$ with the property that there is a unique arc connecting any pair of points $x \neq y \in T$ and it is moreover isometric to the real segment $[0, d(x, y)] \subset \R$. Such an arc is denoted by $[x, y]$ and is called a segment of $T$. We will also use the obvious notation $(x,y) = [x,y] \setminus \{x,y\}$, and sets of this type will sometimes be called intervals. On every real tree $T$, there is a canonical measure $\lambda_T$, called the length measure, whose restriction to every interval is equal to (the pushforward of) Lebesgue measure.
	The operator $I:\F(T)\to L^1(\lambda_T)$ defined by $I\delta(x)=\indicator{[0,x]}$ and extended linearly to $\lipfree{T}$ is a surjective isometry. If $0\in T'\subset T$ is a subtree, then $\lambda_{T'}$ is just the restriction of $\lambda_T$ to $T'$, and $I$ restricts to an isometry between $\F(T')$ and $L^1(\lambda_{T'})=L^1(T',\lambda_T)$. Note also that if a family of segments $[x_i,y_i]\subset T$ intersect pairwise trivially (at one point at most) then so do the supports of $I(m_{x_iy_i})$, and therefore the molecules $m_{x_iy_i}\in\Free(T)$ are an isometric $\ell_1$ basis (of appropriate cardinality). For more information on $\R$-trees and their Lipschitz free spaces, we refer the reader to Godard's article \cite{Godard}. 
	
	In this section, we will use Godard's isometry $G : \F(T) \to L^1(\lambda_T)$ to describe the relatively weakly compact sets in $\F(T)$. To do so, it is necessary to recall first the notion of equi-integrability in $L^1$-spaces. Let $(\Omega , \mathcal A , \mu)$ be a measure space (not necessarily finite or even $\sigma$-finite). Then we say that a set $H \subset L^1(\mu)$ is \textit{equi-integrable} if it is bounded and satisfies the following conditions:
	\begin{enumerate}
		\item[$(a)$] For any $\ep > 0$, there exists $\delta >0$ such that for every $B \in \mathcal A$, $\mu(B) < \delta$ implies that $\int_B |f|\,d\mu < \ep$ for all $f \in H$.
		\item[$(b)$] For any $\ep >0$, there exists $B \in \mathcal A$ with $\mu(B) < \infty$ s.t. $\int_{\Omega\setminus B} |f|\,d\mu < \ep$ for all $f \in H$.
	\end{enumerate}
	Notice that when $\mu$ is finite, condition (a) implies condition (b).   
	According to the Dunford--Pettis theorem, a set $H$ is relatively weakly compact in $L^1(\mu)$ if and only if it is equi-integrable (see e.g. \cite[Theorem~15.4]{Voigt}). It is worth noting that, since $L^1$-spaces are weakly sequentially complete (see \cite[Theorem~5.2.9]{AlbiacKalton}), weakly precompact sets coincide with relatively weakly compact sets. 
	In what follows, we will transfer the equi-integrability conditions through the isometry $L^1(\lambda_T) \to \F(T)$ and describe the resulting properties in $\F(T)$.
	
	First, we shall examine the condition $(a)$ from the definition of equi-integrability above.
	For convenience, we remind the reader that a \emph{convex series of molecules} is a series of the form $\gamma = \sum_{i=1}^{\infty} a_i m_{x_i y_i}$ with moreover $\|\gamma\| = \sum_{i=1}^{\infty} |a_i|$. 
	If the sum is finite, we will say \emph{convex sum of molecules} instead. 
	Convex sums and convex series of molecules should not be confused with convex combinations of molecules.
	We will show that condition $(a)$ above translates into a property on the distribution of the weights $|a_i|$ with respect to the distances $d(x_i,y_i)$ between pairs of points defining the molecules.
	
	\begin{proposition} \label{Charac(a)}
		Let $T$ be a $\R$-tree. We denote by $G : \F(T) \to L^1(\lambda_T)$ Godard's isometry. For a bounded $W \subset \F(T)$, the following are equivalent:
		\begin{enumerate}
			\item[$(i)$] $\forall \ep > 0$, $\exists \delta >0$ such that for all $w \in W$ and for every measurable $B \subset T$, 
			\[\lambda_T(B) < \delta \implies \int_B |Gw|\,d\lambda_T < \ep.\]
			
			\item[$(ii)$] $\forall\ep > 0$, $\exists \delta>0$ such that for every convex sum of molecules $\sum_{i=1}^n a_i m_{x_i  y_i} \in \F(T)$ 
			$\forall I \subset \{1, \ldots , n\}$,
			\[\sum_{i\in I}d(x_i,y_i)<\delta \implies \sum_{i\in I}\abs{a_i}\leq\varepsilon+\dist(\gamma,W).\]
			
			\item[$(iii)$] $\exists C>0$, $\forall \ep >0$, $\exists \delta >0$ such that for all $w \in W$ and for every convex sum of molecules $\sum_{i=1}^n a_i m_{x_i  y_i} \in \F(T)$ with $(x_i)_{i=1}^n\cup(y_i)_{i=1}^n \subset \supp(w)\cup\set{0}$, $\|w - \sum_i a_i m_{x_i y_i}\| \leq C\ep$, $\forall I \subset \{1, \ldots , n\}$, 
			\[\sum_I d(x_i ,y_i) < \delta \implies \sum_I |a_i| \le \ep+C\ep.\]
		\end{enumerate}
	\end{proposition}

	In order to prove $(i) \implies (ii)$ we need the following lemma.
	Observe that for every $x\neq y \in T$ we have $G(m_{xy})=\frac1{d(x,y)}(\indicator{[x,b]}-\indicator{[b,y]})$ where $b\in T$ is the unique point in $[x,y]$ such that $d(0,b)=d(0,[x,y])$.
	\begin{lemma}\label{l:SumOfCoefficients}
		Let $\gamma=\sum_{i=1}^n a_i m_{x_iy_i} \in \Free(T)$ be a convex sum of molecules. 
		Let $I \subset \{1,\ldots,n\}$. 
		Then there exists $f_I\in L_\infty(\lambda_T)$, $\|f_I\|=1$ such that 
		$\duality{f_I,G\gamma}\geq \sum_{i\in I}|a_i|$ and $\supp(f_I) \subset \bigcup_{i\in I} [x_i,y_i]$.
	\end{lemma} 
	\begin{proof}
		We may assume that $a_i\neq 0$ for every $i\in \set{1,\ldots,n}$.
		Then, notice that $\supp(G(a_im_{x_iy_i}))=[x_i,y_i]$. 
		Since $\gamma$ is a convex sum of molecules, \cite[Theorem~2.4]{RRZ}
		yields that $\lambda_T([x_i,y_i]\cap [x_j,y_j])>0$ implies that $\mathrm{sign}(G(a_i m_{x_iy_i}))(t)=\mathrm{sign}(G(a_j m_{x_jy_j}))(t)$ for almost every $t\in [x_i,y_i]\cap [x_j,y_j]$.
		We can thus define $f_I(t):=\mathrm{sign}(G(a_i m_{x_iy_i}))$ whenever $t\in [x_i,y_i]$ for some $i\in I$, and $f_I(t):=0$ otherwise.
		Now 
		$$
		\duality{f_I,G\gamma}=\sum_{i\in I}\abs{a_i}+\sum_{i \notin I} \abs{a_i}\sum_{j\in I}\lambda_T([x_i,y_i]\cap[x_j,y_j]) .
		$$
	\end{proof}
	
	\begin{proof}[Proof of the downward pointing implications of Proposition~\ref{Charac(a)}]
		(i) $\implies$ (ii):
		Let us fix $\varepsilon>0$ and let $\delta>0$ be obtained from (i).
		Let $\gamma\in F(M)$ be a convex sum of molecules.
		Let $I \subset \{1,\ldots,n\}$ such that $\sum d(x_i,y_i)<\delta$.
		By Lemma~\ref{l:SumOfCoefficients} we have for any $w\in W$
		$$
		\begin{aligned}
			\sum_{i\in I}\abs{a_i}&\leq \duality{f_I,G\gamma}\leq \abs{\duality{ f_I,Gw}}+\abs{\duality{ f_I,G\gamma-Gw}}\\
			&\leq \varepsilon+\norm{\gamma-w}
		\end{aligned}
		$$
		where the last inequality follows from the uniform integrability and the fact that $\lambda_T(\supp(f_I))\leq \lambda_T(\bigcup_{i\in I}[x_i,y_i])<\delta$.
		Since this holds for every $w\in W$, we obtain $\sum_{i\in I}\abs{a_i}\leq \varepsilon+\dist(\gamma,W)$. 
		
		(ii) implies clearly $\forall \ep >0$, $\exists \delta >0$ such that for all $C>0$ and for all $w \in W$ and for every convex sum of molecules $\sum_{i=1}^n a_i m_{x_i  y_i} \in \F(T)$, $\|w - \sum_i a_i m_{x_i y_i}\| \leq C\ep$, $\forall I \subset \{1, \ldots , n\}$,         \[\sum_I d(x_i ,y_i) < \delta \implies \sum_I |a_i| \le \ep+C\ep.\]
		Therefore (ii) $\implies$ (iii).
	\end{proof}

	The proof of (iii)$\implies$ (i) uses the following simple convexity argument valid in arbitrary metric space $M$. 
	A sequence $(u_1,u_2,\ldots, u_{k+1}) \subset M$ of pairwise distinct elements of $M$ is called a \emph{partition of the couple of points} $x\neq y \in M$ if $d(x,y)=\sum_{i=1}^k d(u_i,u_{i+1})$, $u_1=x$ and $u_{k+1}=y$.
	Then, we can express the elementary molecule $m_{x y}$ as a convex combination of the smaller molecules:
	$$ m_{xy} = \sum_{i=1}^{k} \dfrac{d(u_i ,u_{i+1})}{d(x,y)} m_{u_i u_{i+1}}.$$
	We also require the following lemma.
	
	\begin{lemma}\label{lem:Char(a)_helper}
		Let $\delta,\varepsilon>0$, $R\ge 1$, and $\gamma:=\sum_{i=1}^{n} a_i m_{x_i y_i}\in\lipfree{M}$ be such that $\sum_{i=1}^{n} \abs{a_i}\le R$ and for every $I \subset \{1,\ldots,n\}$ we have
		\begin{equation*}
			\sum_{i\in I} d(x_i,y_i)<\delta \implies \sum_{i\in I} |a_i| < \varepsilon.
		\end{equation*}
		For every $i\in I$, let $(u_1^i,\ldots,u_{k_i+1}^i)$ be a partition of the couple of points $(x_i,y_i)$. Then
		\begin{equation}\label{e:GammaUnderRefinement}
			\gamma=\sum_{i=1}^{n} a_i\sum_{j=1}^{k_i}\frac{d(u_{j}^i,u_{j+1}^i)}{d(x_i,y_i)} m_{u_j^i u_{j+1}^i}
		\end{equation}
		and for every $\mathcal I\subseteq \set{(i,j):i\in\set{1,\ldots,n},j\in\set{1,\ldots,k_i}}$ we have
		\begin{equation*}
			\sum_{(i,j)\in \mathcal{I}} d(u_{j}^i,u_{j+1}^i)<\frac{\delta\varepsilon}{R} \implies \sum_{(i,j)\in \mathcal{I}} |a_i|\frac{d(u_{j}^i,u_{j+1}^i)}{d(x_i,y_i)} < 3\varepsilon+\frac{4\varepsilon^2}{R}.
		\end{equation*}
	\end{lemma}
	
	\begin{proof}
		The new representation~\eqref{e:GammaUnderRefinement} of $\gamma$ is clear.
		Now, fix $$\mathcal{I}\subseteq \set{(i,j):i\in\set{1,\ldots,n},j\in\set{1,\ldots,k_i}}$$ and assume that 
		\begin{equation*}
			\sum_{(i,j)\in \mathcal{I}} d(u_{j}^i,u_{j+1}^i)<\frac{\delta\varepsilon}{R}.
		\end{equation*}
		
		Let $J=\set{i\in\set{1,\ldots,n}: \exists j\in \set{1,\ldots,k_i+1}, (i,j)\in \mathcal I}$ and let us divide $J$ into pairwise disjoint subsets. Let
		\begin{equation*}
			J_0=\set{i\in J: d(x_i,y_i)\ge\delta}
		\end{equation*}
		and for $m\in\N$ let
		\begin{equation*}
			J_m=\set{i\in J: d(x_i,y_i)<\delta, \frac{1}{2^m}<\sum_{j:(i,j) \in \mathcal I}\frac{d(u_{j}^i,u_{j+1}^i)}{d(x_i,y_i)}\le\frac{1}{2^{m-1}}}.
		\end{equation*}
		Clearly $J_0,J_1,\ldots$ are pairwise disjoint and $J=\bigcup_{m=0}^\infty J_m$. 
		Now for $i\in J_0$ we have 
		\begin{equation*}
			\sum_{j:(i,j) \in \mathcal I}\frac{d(u_{j}^i,u_{j+1}^i)}{d(x_i,y_i)}\le  \frac{1}{\delta}\sum_{j:(i,j) \in \mathcal I}d(u_{j}^i,u_{j+1}^i)\le \frac{1}{\delta}\frac{\delta\varepsilon}{R}=\frac{\varepsilon}{R}.
		\end{equation*}
		Thus
		\begin{equation*}
			\sum_{i\in J_0}\abs{a_i}\sum_{j:(i,j) \in \mathcal I}\frac{d(u_{j}^i,u_{j+1}^i)}{d(x_i,y_i)}\le \frac{\varepsilon}{R}\sum_{i\in J_0}\abs{a_i}\le \varepsilon.
		\end{equation*}
		For $m\in \N$ we have
		\begin{equation*}
			\sum_{i\in J_m}d(x_i,y_i)< 2^m\sum_{i\in J_m}\sum_{j:(i,j) \in \mathcal I}d(u_{j}^i,u_{j+1}^i).
		\end{equation*}
		Note that for every $i\in J_m$ we have $d(x_i,y_i)<\delta$. Thus we can divide the set $J_m$ into $l_m\geq 1$ pairwise disjoint subsets $A_1^m,\ldots,A_{l_m}^m$ such that $J_m=\bigcup_{j=1}^{l_m}A_j^m$, for every $j\in \set{1,\ldots,l_m}$ we have 
		$ \sum_{i\in A_j^m}d(x_i,y_i)<\delta$, and for any $j,j'\in \set{1,\ldots,l_m}$ with $j\neq j'$ we have $\sum_{i\in A_j^m\cup A_{j'}^m}d(x_i,y_i)\ge\delta$. Then 
		\begin{equation*}
			\frac{\delta}{2}(l_m-1) \le \sum_{i\in J_m}d(x_i,y_i)\le  
			2^m\sum_{i\in J_m}\sum_{j:(i,j) \in \mathcal I}d(u_{j}^i,u_{j+1}^i)
		\end{equation*}
		By assumption $\sum_{i\in A_j^m}\abs{a_i}<\varepsilon$ and thus
		\begin{equation*}
			\sum_{i\in J_m}\abs{a_i}=\sum_{j=1}^{l_m}\sum_{i\in A_j^m}\abs{a_i}<l_m\varepsilon.
		\end{equation*}
		Now
		\begin{equation*}
			\sum_{i\in J_m}\abs{a_i}\sum_{j:(i,j) \in \mathcal I}\frac{d(u_{j}^i,u_{j+1}^i)}{d(x_i,y_i)}\le \frac{1}{2^{m-1}}\sum_{i\in J_m}\abs{a_i}< \frac{l_m\varepsilon}{2^{m-1}}.
		\end{equation*}
		Therefore:
		\begin{align*}
			\sum_{(i,j)\in \mathcal I} |a_i|\frac{d(u_{j}^i,u_{j+1}^i)}{d(x_i,y_i)}&=\sum_{i\in J_0} |a_i|\sum_{j:(i,j) \in \mathcal I}\frac{d(u_{j}^i,u_{j+1}^i)}{d(x_i,y_i)}+\sum_{m=1}^\infty\sum_{i\in J_m} |a_i|\sum_{j:(i,j) \in \mathcal I}\frac{d(u_{j}^i,u_{j+1}^i)}{d(x_i,y_i)} \\
			&< \varepsilon+ \sum_{m=1}^\infty\frac{l_m\varepsilon}{2^{m-1}} = 3\varepsilon+ \sum_{m=1}^\infty\frac{(l_m-1)\varepsilon}{2^{m-1}} \\
			&\le 3\varepsilon+\sum_{m=1}^\infty \frac{\varepsilon}{2^{m-1}}\frac{2}{\delta} 2^m\sum_{i\in J_m}\sum_{j:(i,j) \in \mathcal I}d(u_{j}^i,u_{j+1}^i)\\
			&= 3\varepsilon+\frac{4\varepsilon}{\delta}\sum_{m=1}^\infty \sum_{i\in J_m}\sum_{j:(i,j) \in \mathcal I}d(u_{j}^i,u_{j+1}^i) \\
			&\le 3\varepsilon+\frac{4\varepsilon}{\delta} \sum_{(i,j)\in \mathcal I}d(u_{j}^i,u_{j+1}^i)\\
			&\le 3\varepsilon+\frac{4\varepsilon^2}{R}.
		\end{align*}
	\end{proof}
	
	\begin{proof}[Proof of $(iii) \implies (i)$ of Proposition \ref{Charac(a)}] 
		
		Let $C>0$ be obtained by $(iii)$. 
		Let $\ep>0$ be fixed and $\delta \in (0,1)$ be obtained by$(iii)$. We set $R=\sup_{w\in W}\|w\|+C\varepsilon$.  
		Consider a measurable subset $B \subset T$ with $\lambda_T(B) < \frac{\delta(\varepsilon+C\varepsilon)}{R}$. We will find a bound for $\int_B |I(w)|\,d\lambda_T$; to do so, by definition of the length measure (see \cite{Godard}), we may assume that $B$ is a finite union of segments, $B = \bigcup_{k=1}^{N} [a_k,b_k]$, with the sets $(a_k,b_k)$ being pairwise disjoint.  Pick any $w \in W$ and let $\gamma = \sum_{i=1}^{n} a_i m_{x_i y_i}\in\F(\supp(w)\cup\set{0})$ be a convex sum of molecules such that $\|w - \gamma\| \leq C\ep$.
		Notice that $\sum_{i=1}^n \abs{a_i}=\norm{\gamma}\leq R$.
		We may assume that all segments $[x_i,y_i]$ are disjoint except for possibly the endpoints, by subdividing each segment if necessary, using the convexity argument outlined earlier, and then grouping and adding equal terms together. Similarly, we assume that the sets $[x_i,y_i]$ and $[a_k,b_k]$ either coincide or satisfy $\lambda_T([x_i,y_i] \cap [a_k,b_k])=0$ (at most finitely many points in the intersection).
		The new representation of $\gamma$ might contain molecules whose endpoints are not in the set $\supp(w)\cup\set{0}$, so we must enlist the help of Lemma \ref{lem:Char(a)_helper} to see that for every $I\subseteq \set{1,\ldots,n}$ we have
		\begin{equation*}
			(\star) \quad
			\sum_I d(x_i ,y_i)<\frac{\delta(\varepsilon+C\varepsilon)}{R} \implies \sum_I |a_i| < 3(\varepsilon+C\varepsilon)+\frac{4(\varepsilon+C\varepsilon)^2}{R}.
		\end{equation*}
		From here, some direct calculations will yield the desired result:
		\begin{align*}
			\int_B |G(w)|\,d\lambda_T &\leq C\ep\lambda_T(B) +\sum_{k=1}^N \int_{a_k}^{b_k}|G(\gamma)|\,d \lambda_T\\ &\leq C\ep \lambda_T(B) +  \sum_{k=1}^N \sum_{i=1}^{n} \int_{a_k}^{b_k} |a_i| \dfrac{\indicator{[x_i,y_i]}}{d(x_i,y_i)} \, d \lambda_T \\
			&<  C\ep + \sum_{k=1}^N \sum_{i=1}^{n} \left\{\begin{array}{cc}
				|a_{i}| & \text{ if } [x_i,y_i] = [a_k , b_k] \\
				0 & \text{ otherwise}
			\end{array}
			\right. \\
			&\leq 3\ep+4C\varepsilon+\frac{4(\varepsilon+C\varepsilon)^2}{R}.
		\end{align*}
		The very last inequality follows from
		$$\sum_{i \,:\, \exists k, [x_i , y_i]=[a_k,b_k]} d(x_i,y_i) \leq \sum_{k=1}^N d(a_k,b_k) \leq \frac{\delta(\varepsilon+C\varepsilon)}{R}$$
		and thus property $(\star)$ provides the last estimate.
	\end{proof}

	Next we aim to give other equivalent conditions for $(ii)$ and $(iii)$. To that end, we will benefit from the following definition.
	\begin{definition}
		Let $M$ be a metric space.
		For $\delta>0$ we define the set     
		\begin{equation*}
			\SSM{\delta}{M}=\set{\sum_{i=1}^n a_i m_{x_i  y_i} \in \F(M) : \sum_{i=1}^n d(x_i  y_i)<\delta, \sum_{i=1}^n \abs{a_i}= \left\|\sum_{i=1}^n a_i m_{x_i  y_i}\right\|}.
		\end{equation*}
	\end{definition}

	\begin{remark}
		Note that, at least for $\R$-trees, we have
		\begin{equation*}
			\SSM{\delta}{T}=\set{\sum_{i=1}^n a_i m_{x_i  y_i} \in \F(T) : \sum_{i=1}^n  d(x_i,  y_i)<\delta}.
		\end{equation*}
		Indeed, let $\sum_{i=1}^n a_i m_{x_i  y_i} \in \F(M)$ be such that $\sum_{i=1}^n  d(x_i,  y_i)<\delta$. We may assume that all segments $[x_i,y_i]$ either coincide or their intersections contain a finite amount of points, by subdividing each segment if necessary and using the convexity argument outlined earlier. Note that in this process the sum $\sum_{i=1}^n d(x_i,  y_i)$ does not change. Next we add together the equal terms. The sum can decrease in this process, but not increase, and thus for the new representation we still have the desired inequality. Furthermore, the new representation is a convex sum of molecules (even an $\ell_1$ sum, see the comment at the start of Section \ref{section:Trees}).
	\end{remark}
	
	The set $\SSM{\delta}{M}$ will be used in the following proposition. 
	\begin{proposition}\label{Char(a)_equivalent}
		Let $T$ be a $\R$-tree. For a bounded $W \subset \F(T)$, the following are equivalent:
		\begin{enumerate}
			\item[$(i)$] $\forall C>0$, $\forall \ep >0$, $\exists \delta >0$ such that for all $w \in W$ and for every convex sum of molecules $\sum_{i=1}^n a_i m_{x_i  y_i} \in \F(T)$, $\|w - \sum_i a_i m_{x_i y_i}\| \leq C\ep$, $\forall I \subset \{1, \ldots , n\}$, $\sum_I d(x_i ,y_i) < \delta$ implies that $\sum_I |a_i| \le \ep+C\ep$.
			\item[$(ii)$] $\exists C>0$, $\forall \ep >0$, $\exists \delta >0$ such that for all $w \in W$ and for every convex sum of molecules $\sum_{i=1}^n a_i m_{x_i  y_i} \in \F(T)$, $\|w - \sum_i a_i m_{x_i y_i}\| \leq C\ep$, $\forall I \subset \{1, \ldots , n\}$, $\sum_I d(x_i ,y_i) < \delta$ implies that $\sum_I |a_i| \le \ep+C\ep$.
			\item[$(iii)$] $\forall C>0$, $\forall \ep >0$, $\exists \delta >0$ such that for all $w \in W$ and for every convex sum of molecules $\sum_{i=1}^n a_i m_{x_i  y_i} \in \F(\supp(w)\cup\set{0})$, $\|w - \sum_i a_i m_{x_i y_i}\| \leq C\ep$, $\forall I \subset \{1, \ldots , n\}$, $\sum_I d(x_i ,y_i) < \delta$ implies that $\sum_I |a_i| \le \ep+C\ep$.
			\item[$(iv)$] $\exists C>0$, $\forall \ep >0$, $\exists \delta >0$ such that for all $w \in W$ and for every convex sum of molecules $\sum_{i=1}^n a_i m_{x_i  y_i} \in \F(\supp(w)\cup\set{0})$, $\|w - \sum_i a_i m_{x_i y_i}\| \leq C\ep$, $\forall I \subset \{1, \ldots , n\}$, $\sum_I d(x_i ,y_i) < \delta$ implies that $\sum_I |a_i| \le \ep+C\ep$.
			\item[$(v)$] $\forall \ep >0$, $\exists \delta >0$ such that $\dist\left(w,\SSM{\delta}{T}\right)\ge \norm{w}-\varepsilon$  for every $w\in W$.
			\item[$(vi)$] $\forall \ep >0$, $\exists \delta >0$ such that $\dist\left(w,\SSM{\delta}{\supp(w)\cup\set{0}}\right)\ge \norm{w}-\varepsilon$ for every $w \in W$.
		\end{enumerate}
	\end{proposition}
	
	\begin{proof}
		The implications $(i) \implies (ii) \implies (iv)$, $(i) \implies (iii) \implies (iv)$, and $(v) \implies (vi)$ are obvious, and the implication $(iv) \implies (i)$ follows from Proposition \ref{Charac(a)}.
		
		\medskip
		$(vi) \implies (iv)$: We will prove the contrapositive.
		Assume that $(iv)$ fails. Fix $C<1$ and obtain $\ep>0$ witnessing the failure of $(iv)$. 
		Fix $\delta>0$ and find $w\in W$, a convex  sum of molecules $\sum_{i=1}^n a_i m_{x_i  y_i} \in \F(\supp(w)\cup\set{0})$, and $I \subset \{1, \ldots , n\}$ such that $\|w - \sum_i a_i m_{x_i y_i}\| \leq C\ep$, $\sum_I d(x_i ,y_i) < \delta$, and $\sum_I |a_i| > \ep+C\ep$. Then 
		$\sum_I a_i m_{x_i y_i}\in \SSM{\delta}{\supp(w)\cup\set{0}}$ and 
		\begin{align*}
			\norm{w-\sum_I a_i m_{x_i y_i}}&\le \norm{w-\sum_{i=1}^n a_i m_{x_i y_i}}+\sum_{i\notin I}\abs{a_i}\\
			&\le C\varepsilon +\sum_{i=1}^n \abs{a_i}-\sum_I \abs{a_i}< C\varepsilon+\norm{\sum_{i=1}^n a_i m_{x_i y_i}}-(\varepsilon+C\varepsilon)\\
			&\le C\varepsilon+\norm{w}+C\varepsilon-(\varepsilon+C\varepsilon)=\norm{w}-(1-C)\varepsilon.
		\end{align*}
		Hence we have now shown that for $\varepsilon':=(1-C)\varepsilon$ and for every $\delta >0$ there exists $w\in W$ such that $$\dist\left(w,\SSM{\delta}{\supp(w)\cup\set{0}}\right)< \norm{w}-\varepsilon',$$ i.e. the negation of the condition $(vi)$.
		
		\medskip
		$(i) \implies (v)$:
		We will prove the contrapositive. Assume that 
		$(v)$ fails and let $\ep>0$ witness its failure. 
		Fix $\delta>0$ and let $w\in W$ be such that $\dist\left(w,\SSM{\delta}{T}\right)< \norm{w}-\varepsilon$. Let $\gamma:=\sum_{i=1}^n a_i m_{x_i  y_i}\in \SSM{\delta}{T}$ be such that $\sum_{i=1}^n d(x_i,y_i)<\delta$ and $\norm{w-\gamma}<\norm{w}-\varepsilon$. 
		Let $\sum_{i=1}^k b_i m_{u_i  v_i}\in \F(T)$ be a convex sum of molecules such that 
		$$\norm{w-\sum_{i=1}^k b_i m_{u_i  v_i}}<\frac{\varepsilon}{4}.$$ Then $\norm{\gamma-\sum_{i=1}^k b_i m_{u_i  v_i}}<\norm{w}-\frac{3\varepsilon}{4}$. By using the convexity argument introduced earlier, we may assume that $[u_i,v_i]$ and $[x_j,y_j]$ either coincide or are disjoint, except possibly at endpoints, and that all pairs $(u_i,v_i)$ are different.
		Let 
		\begin{equation*}
			I:=\set{i\in\set{1,\ldots,k}:\exists j\in\set{1,\ldots,n}, [u_i,v_i]=[x_j,y_j]}.
		\end{equation*}
		Then $\sum_Id(u_i,v_i)\le \sum_{i=1}^n d(x_i,y_i)<\delta$ and
		\begin{equation*}
			\norm{w}-\frac{3\varepsilon}{4}>\norm{\gamma-\sum_{i=1}^k b_i m_{u_i  v_i}} \ge \sum_{i\notin I}\abs{b_i} = \sum_{i=1}^k\abs{b_i} -\sum_I\abs{b_i}>\norm{w}-\frac{\varepsilon}{4}- \sum_I\abs{b_i},
		\end{equation*}
		where the second inequality follows from the fact that molecules with disjoint support are in $\ell_1$ sum, see the start of Section \ref{section:Trees}. This yields $\sum_I\abs{b_i}>\frac{\varepsilon}{2}$. As $\delta>0$ was arbitrary, we now get a contradiction with condition $(i)$ when picking $C=1$ and $\varepsilon=\frac{\varepsilon}{4}$.
	\end{proof}
	
	Several of these implications actually hold for arbitrary metric spaces:
	$(i)\implies(ii)\implies(iv)$, $(i)\implies(iii)\implies(iv)$,
	$(v)\implies(vi)\implies(iv)$ and $(v)\implies(ii)$.
	This leads to the following question.
	
	\begin{question}
		Do the equivalences in Proposition \ref{Char(a)_equivalent} hold for arbitrary metric spaces? Does weak pre-compactness imply any (or all) of them?
	\end{question}

	\bigskip

	We now turn to condition $(b)$ in the definition of equi-integrability mentioned above. It corresponds to being ``almost finitely generated'' in the following sense.
	
	\begin{proposition} \label{Charac(b)}
		Let $T$ be an $\R$-tree. 
		For a bounded $W \subset \F(T)$, the following are equivalent:
		\begin{enumerate}[$(i)$]
			\item For any $\ep >0$, there exists a measurable set $B \subset T$ with $\lambda_T(B) < \infty$ and $\int_{T\setminus B} |Gw|\,d\lambda_T < \ep$ for all $w \in W$.
			\item For any $\ep >0$, there exists a finitely generated subtree $T':=\bigcup_{i=1}^n [0,z_i] \subset T$ such that $W \subset \F(T') + \ep B_{\F(T)}$.
		\end{enumerate}
	\end{proposition}
	
	\begin{proof}
		$(i) \implies (ii)$: Fix $\ep > 0$. Let $B$ be a measurable subset as in $(i)$. By the definition of $\lambda_T$,  we may assume that $B$ is a finite union of disjoint segments: $B = \bigcup_{i=1}^{N} S_i$. Let $T'$ be the subtree of $T$ generated by $0$ and the segments $(S_i)_{i=1}^N$. The resulting tree $T'$ can be expressed as $T'= \bigcup_{i=1}^n [0,z_i]$. 
		
		Now, consider $w \in W$. 
		By the assumption, we have: 
		$$\int_{T} |Gw| \,d\lambda_T -\int_{T} \indicator{B}|Gw| \, d\lambda_T = \int_{T\setminus B} |Gw| \,d\lambda_T < \ep.$$
		Thus, if we define $w':= G^{-1}(\indicator{B} Gw)$, then we have that $w' \in \F(T')$ and:
		\begin{align*}
			\|w - w'\| &= \|Gw - Gw'\| = \int_T |Gw - Gw'| \,d\lambda_T \\
			& =  \int_{T } |Gw - \indicator{B}Gw| \,d\lambda_T = \int_{T\setminus B} |Gw| \,d\lambda_T < \ep.
		\end{align*}
		\medskip
		
		$(ii) \implies (i)$: Fix $\ep > 0$. Let $T'$ be a finitely generated subtree as in $(ii)$. Setting $B = T'$, we have a measurable subset of $T$ with $\lambda_T(B) < \infty$. Consider $w \in W$. By the assumption, there exists $w' \in \F(B)$ such that $\|w - w'\| \leq \ep$.  The following calculation completes the proof:
		\begin{align*}
			\int_{T\setminus B} |Gw|\,d\lambda_T = \int_{T\setminus B} |Gw-Gw'|\,d\lambda_{T} \leq \int_{T} |Gw-Gw'|\,d\lambda_{T} =\|w -w'\| \leq \ep.
		\end{align*}
	\end{proof}
	
	Combining Proposition~\ref{Charac(a)} with Proposition~\ref{Charac(b)} yields the next characterization (Theorem \ref{thmC}).
	
	\begin{corollary}\label{cor:char_R_tree} 
		Let $T$ be an $\R$-tree. If $W \subset \F(T)$ then $W$ is relatively weakly compact if and only if it is bounded and satisfies: 
		\begin{enumerate}[$(a)$,leftmargin=*]
			\item $\forall\ep > 0$, $\exists \delta>0$ such that for every convex sum of molecules $\sum_{i=1}^n a_i m_{x_i  y_i} \in \F(T)$, 
			$\forall I \subset \{1, \ldots , n\}$,
			\[\sum_{i\in I}d(x_i,y_i)<\delta \implies \sum_{i\in I}\abs{a_i}\leq\varepsilon+\dist(\gamma,W).\]
			\item For any $\ep >0$,  there exists a finitely generated subtree $T':=\bigcup_{i=1}^n [0,z_i] \subset T$ such that $W \subset \F(T') + \ep B_{\F(T)}$.
		\end{enumerate}
	\end{corollary}
	
	\begin{remark}
		The condition (\textit{a}) can be replaced with any of the equivalent conditions from Proposition \ref{Char(a)_equivalent}.
	\end{remark}
	
	\begin{remark}
		While condition (\textit{a}) in Corollary~\ref{cor:char_R_tree} is meaningful in any metric space, condition (\textit{b}) is intrinsically tied to the tree structure of the underlying space. From the perspective of potential generalizations, it would be valuable to identify an alternative condition that does not explicitly rely on the tree structure of $T$. In particular, it remains unclear whether condition (\textit{b}) can be substituted by uniform regularity. 
	\end{remark}
	
	\begin{question} \label{Question:TreebUR}
		Can condition $(b)$ in Corollary~\ref{cor:char_R_tree} be replaced with uniform regularity?
	\end{question}
	
	The next example shows that (\textit{b}) cannot always be replaced by tightness.
	
	\begin{example}
		Let $M=\{0\}\cup \bigcup_{n=1}^\infty \{x^n_1,\ldots,x^n_n\}$ with a tree metric $d$ such that $0$ is unique branching point and each non-zero point is sitting on its own branch with $d(x^n_i,0)=\frac1n$.
		Let $\mu_n=\frac1n(m_{x^n_10}+\ldots+m_{x^n_n0})=\delta(x^n_1)+\ldots+\delta(x^n_n)$. Let $T$ be the smallest $\R$-tree which contains $M$.  
		Then $W:=\set{\mu_n : n \in M}$ is isometrically equivalent to the $\ell_1$-basis, and therefore $W$ is not weakly compact. Notice also that $W$ is tight since it is bounded in $\F(M)$ and $M$ is compact. Next, we will show that condition $(vi)$ from Proposition \ref{Char(a)_equivalent}, and thus condition $(a)$ from Corollary~\ref{cor:char_R_tree}, is satisfied. To that end, fix $\varepsilon>0$ and choose $\delta=\varepsilon$. Fix $\mu_n\in W$ and a convex sum of molecules $\gamma:=\sum_{i=1}^ka_im_{y_iz_i}\in \SSM{\delta}{\supp(\mu_n)\cup\set{0}}$. Without loss of generality we may assume that for every $i\in \set{1,\ldots,k}$ either $y_i=0$ or $z_i=0$, and that each $x_i^n$ appears once at most. Let $I:=\set{i\in\set{1,\ldots,n}: x_i^n\in\set{y_1,z_1,\ldots,y_k,z_k}}$. Then 
		\begin{equation*}
			\frac{\abs{I}}{n}=\sum_Id(y_i,z_i)\le \sum_{i=1}^kd(y_i,z_i)<\delta
		\end{equation*}
		and 
		\begin{equation*}
			\norm{\mu_n-\sum_{i=1}^ka_im_{y_iz_i}}\ge \sum_{i\notin I} \frac{1}{n}=1-\frac{\abs{I}}{n}>1-\delta=1-\varepsilon.
		\end{equation*}
		
		This shows that tightness and condition~$(a)$ are not sufficient for weak compactness. 
		However, in this specific example, condition~$(b)$ in Corollary~\ref{cor:char_R_tree} can be replaced by uniform regularity. 
		Indeed, let $W \subset \F(T)$ and $\varepsilon > 0$. 
		Applying uniform regularity to the open set $U = M \setminus \{0\}$, we obtain a compact set $K \subset U$ such that 
		$$
		W \subset \F(K \cup \{0\}) + \varepsilon B_{\F(M)}.
		$$
		Since $K \subset T$ is compact, there exists a finitely generated subtree $T'$ such that $K \cup \{0\} \subset T'$. 
		Hence, uniform regularity implies condition~$(b)$. 
		Finally, if $W$ satisfies conditions~$(a)$ and~$(b)$, then it is relatively weakly compact, and therefore uniformly regular by Corollary~\ref{thmVstarUR}.
		The same argument goes through for trees with finitely many branching points $\mathrm{Br}(T)$. 
		In that case take $U=M\setminus\mathrm{Br}(T)$ and proceed as above.
	\end{example}

	\section{Sequences of sums of ``small molecules'' contain an \texorpdfstring{$\ell_1$}{l1}-basis}
	\label{section:smallmolecules}
	
	The previous section highlights the significance of approximation by convex series of molecules. Informally, within a weakly compact set, the weight of molecules formed by pairs of points in close proximity should be controlled in a uniform way. In this section, we establish that this condition is, in a sense, necessary. Specifically, consider a sequence of (finite or countable) linear combinations of molecules, denoted by $\gamma_n=\sum_k \lambda_k^n m_{x_k^n y_k^n}$, satisfying $\sum_k d(x_k^n,y_k^n)\to 0$ as $n \to \infty$ and $\sup_n\sum_k |\lambda_k^n|<\infty$. We will demonstrate that this sequence contains an $\ell_1$-basis, thereby proving that such a sequence cannot be contained in a weakly precompact set.

	We shall start by introducing a definition of $\delta$-cyclical monotonicity between subsets of $\widetilde{M}:=\set{(x,y): x,y\in M, x\neq y}$. For a subset of $\widetilde{M}$, the concept of cyclical monotonicity can be found e.g.
	in \cite{RRZ}: 
	the set $A\subset \widetilde{M}$ is
	\textit{cyclically monotonic} if,
	for every finite sequence  of 
	pairs $(x_1,y_1),\ldots,(x_n,y_n) \in A$, we have
	\[
	\sum_{i=1}^n d(x_i,y_{i+1})\geq \sum_{i=1}^n d(x_i,y_i),
	\]
	where $y_{n+1}=y_1$.
	This concept was generalized in \cite{HKO} by introducing a parameter $\delta\in(0,1]$: a subset $A\subset \widetilde{M}$ is \emph{$\delta$-cyclically monotonic}, if for any finite sequence of pairs $(x_1,y_1)$, $\ldots$ , $(x_n,y_n)\in A$, we have
	\begin{equation*}
		\sum_{i=1}^n\min\set{d(x_i,y_{i+1})-\delta d(x_i,y_i),d(y_i,y_{i+1})}\ge 0,
	\end{equation*}
	where $y_{n+1}:=y_1$. Note that $A$ is cyclically monotonic if and only if it is 1-cyclically monotonic. 
	The motivation for both \cite{RRZ} and~\cite{HKO} to consider these definitions 
	was to find a lower estimate for the norm of the sum of molecules. In particular, by \cite[Proposition 2.2]{HKO} (or by \cite[Theorem~2.4]{RRZ} for the case $\delta=1$), a set $A\subset\widetilde{M}$ is $\delta$-cyclically monotonic if and only if there exists $f\in S_{\Lip_0(M)}$ such that $\<f,m_{xy}\>\ge \delta$ for every $(x,y)\in A$.

	\begin{definition}
		Let $A,B\subset \widetilde{M}$ and let $\delta\in (0,1]$. We say that the sets $A$ and $B$ are \emph{$\delta$-cross-cyclically monotonic} if for  any finite sequence of pairwise distinct pairs $(x_1,y_1),\ldots,(x_n,y_n)\in A\cup B$, which contains at least one pair from $A$ and one pair from $B$, we have
		\begin{equation}\label{ineq_cyclical_monotonicity_between_sets}
			\sum_{i=1}^n\min\set{d(x_i,y_{i+1})-\delta d(x_i,y_i),d(y_i,y_{i+1})}\ge 0,
		\end{equation}
		where $y_{n+1}:=y_1$. When the sets $A$ and $B$ are 1-cross-cyclically monotonic, the inequality \eqref{ineq_cyclical_monotonicity_between_sets} reduces to
		\begin{equation*}
			\sum_{i=1}^n\left(d(x_i,y_{i+1})- d(x_i,y_i)\right)\ge 0,
		\end{equation*}
		and in this case, we simply say that the sets $A$ and $B$ are \emph{cross-cyclically monotonic}.
	\end{definition}
	
	\begin{remark}
		Notice that when $A$ and $B$ are $\delta$-cross-cyclically monotonic  (with $\delta>0$), then $(x,y)\in A$ implies $(y,x)\notin B$. Nevertheless this does not prevent both $(x,y)$ and $(y,x)$ from belonging to $A$. Indeed, let $M=\set{x,y,a,b}$ such that $d(x,y)=d(a,b)=1$ and every other non-zero distance is $2$. Then $A=\set{(x,y),(y,x)}$ and $B=\set{(a,b),(b,a)}$ are cross-cyclically monotonic, since for any sequence of pairwise distinct elements $(x_1,y_1),\ldots,(x_n,y_n)\in A\cup B$, $n=2,3,4$, that contains elements from both sets, we have
		\begin{equation*}
			\sum_{i=1}^nd(x_i,y_{i+1})\ge4\ge \sum_{i=1}^nd(x_i,y_i),
		\end{equation*}
		where $y_{i+1}=y_1$. Note that $A$ and $B$ are clearly not $\delta$-cyclically monotonic for any $\delta>0$. In particular, cross-cyclical monotonicity places no cyclical-monotonicity constraint on the individual components. 
	\end{remark}
	
	\begin{remark}\label{r:TwoToOne}
		We need the assumption that the pairs $(x_1,y_1),\ldots,(x_n,y_n)$ are pairwise distinct to avoid the situation when we select the same pair twice from the same set. However, it is straightforward to check that, when the sets $A$ and $B$ are cross-cyclically monotonic, then the inequality \eqref{ineq_cyclical_monotonicity_between_sets} holds also for all such sequences $(x_1,y_1),\ldots,(x_n,y_n)\in A\cup B$, where 
		whenever two pairs $(x_k,y_k)$ and $(x_{k'},y_{k'})$ are equal, then they belong to different sets $A$ and $B$.
	\end{remark}
	
	We now extend \cite[Theorem 2.4]{RRZ} and \cite[Proposition 2.2]{HKO} to the setting of mutually $\delta$-cyclically monotonic sets.
	
	\begin{proposition}\label{prop:Cyclical_monotonicity_between_sets}
		Let $\delta \in (0,1]$ and let $A_1,A_2\subset \widetilde{M}$ be $\delta$-cross-cyclically monotonic.
		Let $\lambda^i \in \ell_1^+(A_i)$, and denote $\gamma_i=\sum_{(x,y)\in A_i} \lambda^i_{(x,y)}m_{xy}$, for $i=1,2$.
		Then there exists $h\in S_{\Lip_0(M)}$ such that $\duality{h,\gamma_i}\geq \delta \norm{\gamma_i}$ for $i=1,2$.
		In particular $\norm{\gamma_1+\gamma_2}\geq \delta(\norm{\gamma_1}+\norm{\gamma_2})$.
	\end{proposition}
	
	In the proof, we are going to use the following lemma, where the equivalence (i)$\iff$(ii) is covered by \cite[Lemma~2.2]{RRZ} (for infinite sequences) and \cite[Lemma~2.3]{RRZ} (for finite sequences), and the equivalence (ii)$\iff$(iii) is straightforward to check, and thus left to the reader.
	\begin{lemma}\label{l:alpha-beta}
		Let $I$ be a set and $\beta_{ij}$, $i,j\in I$, be real numbers such that $\beta_{ii}=0$ for every $i\in I$. 
		Then the following statements are equivalent:
		\begin{itemize}
			\item[(i)] There exist numbers $\alpha_i$, $i\in I$ such that 
			\[
			\alpha_i\leq \alpha_j+\beta_{ij}
			\]
			for $i,j\in I$.
			
			\item[(ii)] For every finite sequence $i_1,\ldots,i_m$ in $I$ we have
			\[
			\beta_{i_1i_2}+\beta_{i_2i_3}\ldots+\beta_{i_mi_1}\geq 0,
			\]
			\item[(iii)]  For every \emph{injective} finite sequence $i_1,\ldots,i_m$ in $I$ we have
			\[
			\beta_{i_1i_2}+\beta_{i_2i_3}\ldots+\beta_{i_mi_1}\geq 0,
			\]
		\end{itemize}

	\end{lemma}
	
	\begin{proof}[Proof of Proposition~\ref{prop:Cyclical_monotonicity_between_sets}]
		The proof follows similar ideas as \cite[Theorem 2.4]{RRZ} and \cite[Proposition 2.2]{HKO}.
		Let $f,g\in S_{\Lip_0(M)}$ be such that $\<f,\gamma_1\>=\norm{\gamma_1}$ and $\<g,\gamma_2\>=\norm{\gamma_2}$. 
		Write $A_1=\set{(x_i,y_i):i\in I_1}$ and $A_2=\set{(x_i,y_i):i\in I_2}$, 
		where $I_1$ and $I_2$ are disjoint sets of indices. Let $I=I_1\cup I_2$. 
		For $i,j\in I$, let 
		\begin{equation*}
			\beta_{ij}:=\begin{cases}
				\min\set{d(x_i,y_j)-\delta\left(f(x_i)-f(y_i)\right),d(y_i,y_j)}, &\text{ if }i\in I_1 \\
				\min\set{d(x_i,y_j)-\delta\left(g(x_i)-g(y_i)\right),d(y_i,y_j)}, &\text{ if }i\in I_2.
			\end{cases}
		\end{equation*}
		Clearly $\beta_{jj}=0$ for all $j\in I$. Let us consider pairwise distinct indices $i_1,\ldots,i_n\in I$. 
		Our aim is to show that 
		\begin{equation*}
			\sum_{k=1}^n \beta_{i_ki_{k+1}}\ge 0,
		\end{equation*}
		where we denote $i_{n+1}:=i_1$. 
		First consider the case when $i_1,\ldots,i_n\in I_1$. Then
		\begin{align*}
			\sum_{k=1}^n \beta_{i_ki_{k+1}}&=\sum_{k=1}^n \min\set{d(x_{i_k},y_{i_{k+1}})-\delta\left(f(x_{i_k})-f(y_{i_k})\right),d(y_{i_k},y_{i_{k+1}})}\\
			&\ge\delta\sum_{k=1}^n \min\set{d(x_{i_k},y_{i_{k+1}})-f(x_{i_k})+f(y_{i_k}),d(y_{i_k},y_{i_{k+1}})}\\
			&\ge\delta\sum_{k=1}^n \left(-f(y_{i_{k+1}})+f(y_{i_k})\right)=0.
		\end{align*}
		The case when $i_1,\ldots,i_n\in I_2$ is analogous. Now assume that  $i_1,\ldots,i_n$ contain indices from both sets $I_1$ and $I_2$.
		For all $i,j\in I$, we have 
		\begin{equation*}
			\beta_{ij}\ge \min\set{d(x_i,y_j)-\delta d(x_i,y_i),d(y_i,y_j)},
		\end{equation*}
		and thus we obtain $\sum_{k=1}^n \beta_{i_ki_{k+1}}\ge 0$ from the fact that the sets $A_1$ and $A_2$ are $\delta$-cross-cyclically monotonic 
		and from Remark~\ref{r:TwoToOne}.
		Indeed, notice that if for $i,j \in I$ with $i\neq j$ we have $(x_i,y_i)=(x_j,y_j)$ then necessarily $(x_i,y_i)\in A_1$ and $(x_j,y_j)\in A_2$ or vice versa.
		Next, from 
		Lemma~\ref{l:alpha-beta}
		we obtain real numbers $\alpha_i$, $i\in I$ such that 
		\begin{equation*}
			\alpha_i\le \alpha_j+\beta_{ij}
		\end{equation*}
		for all $i,j\in I$. Define $h:M\rightarrow\R$ by
		\begin{equation*}
			h(u)=\inf_{i\in I}(\alpha_i+d(u,y_i)).
		\end{equation*}
		As the infimum of 1-Lipschitz functions, $h$ is 1-Lipschitz. 
		The properties of $\alpha_i$, $i\in I$, are not affected by adding a constant, thus we may additionally assume they were chosen in such a way that $h(0)=0$ and thus $h\in B_{\Lip_0(M)}$. 
		It remains to show that $\<h,\gamma_i\>\ge \delta\norm{\gamma_i}$, for $i=1,2$. 
		Fix $i\in I_1$ and $\varepsilon>0$. 
		There exists $j\in I$ such that $h(x_i)\ge\alpha_j+d(x_i,y_j)-\varepsilon$. Furthermore, $h(y_i)\le \alpha_i$, and thus 
		\begin{equation*}
			h(x_i)-h(y_i)\ge \alpha_j+d(x_i,y_j)-\varepsilon-\alpha_i\ge d(x_i,y_j)-\beta_{ij}-\varepsilon\ge \delta\left(f(x_i)-f(y_i)\right)-\varepsilon.
		\end{equation*}
		As $\varepsilon>0$ was arbitrary, we obtain $ h(x_i)-h(y_i)\ge \delta\left(f(x_i)-f(y_i)\right)$ for every $i\in I_1$, and thus $\<h,\gamma_1\>\ge \delta\<f,\gamma_1\>=\delta\norm{\gamma_1}$. 
		By the same argument, we obtain $ h(x_i)-h(y_i)\ge \delta\left(g(x_i)-g(y_i)\right)$ for every $i\in I_2$ and thus $\<h,\gamma_2\>\ge\delta \<g,\gamma_2\>=\delta\norm{\gamma_2}$. 
		This concludes the proof.
	\end{proof}
	
	\begin{lemma}\label{lem:conv_comb}
		Let $M$ be a metric space.
		Let $\mu:=\sum_k\lambda_k m_{x_ky_k}\in \F(M)$ where $(\lambda_k)\in \ell_1$   and let $\gamma\in \F(M)$. Let  $\delta:=\sum_k d(x_k, y_k)$ and
		\begin{equation*}
			C:=\inf\{d(x,y): \, x\in \supp(\gamma)\cup\set{0}, \, y\in \set{x_1,y_1, x_2,y_2,\ldots}  \},
		\end{equation*}
		If $C>\delta$, then $\norm{\gamma+ \mu}\ge \frac{C}{C+\delta}(\|\gamma\| + \|\mu\|)$.
	\end{lemma}
	
	\begin{proof}
		By an obvious approximation argument, we may assume that $\gamma$ and $\mu$ are finitely supported, with $\mu$ written as $\mu:=\sum_{k=1}^n\lambda_k m_{x_ky_k}$. We may additionally assume $\lambda_k>0$ by interchanging $x_k$ and $y_k$ whenever necessary.
		Let $A=\supp(\gamma)\cup\set{0}$ and $B=\set{x_1,y_1,\ldots,x_n,y_n}$. We may further assume that $\gamma$ is written as a convex sum of molecules from points in $A$, say $m_{x_{n+1}y_{n+1}},\ldots,m_{x_Ny_N}$, where $x_{n+1},y_{n+1},\ldots,x_N,y_N\in A$.
		Finally, we assume without loss of generality that $M=A\cup B$.
		Note that $A\cap B=\emptyset$ as $d(A,B)=C>\delta>0$.
		
		Define a new metric $b$ on $M$ by
		$$
		b(x,y) = \begin{cases}
			d(x,y) &\text{ if $x,y\in A$ or $x,y\in B$} \\
			d(x,y)+\delta &\text{ otherwise.}
		\end{cases}
		$$
		To prove that it is indeed a metric, it is enough to check the triangle inequality. Let $x,y,z\in M$ be different. If $x,z$ are in the same set $A$ or $B$ then
		$$
		b(x,z)=d(x,z)\leq d(x,y)+d(y,z)\leq b(x,y)+b(y,z).
		$$
		Otherwise, $y$ must be in the same set as exactly one of them, say $x$, then
		$$
		b(x,z)=d(x,z)+\delta\leq d(x,y)+d(y,z)+\delta = b(x,y)+b(y,z).
		$$
		So $b$ is indeed a metric on $M$ and, moreover, it is equivalent to $d$: if $d(x,y)\neq b(x,y)$ then $d(x,y)\geq C$, thus
		$$
		b(x,y)=d(x,y)+\delta\leq \pare{1+\frac{\delta}{C}}d(x,y) .
		$$
		Therefore $d\leq b\leq (1+\delta/C)d$. In particular $\lipfree{M,d}$ and $\lipfree{M,b}$ are $(1+\delta/C)$-isomorphic.
		\smallskip
		
		\noindent \textbf{Claim:} \textit{The sets $\set{(x_1,y_1),\ldots,(x_n,y_n)}$ and $\set{(x_{n+1},y_{n+1}),\ldots,(x_N,y_N)}$} are cross-cyclically monotonic in the metric $b$.
		\smallskip
		
		Let us assume that the claim is proved. 
		By Proposition \ref{prop:Cyclical_monotonicity_between_sets} we get 
		$$\norm{\mu+\gamma}_{\F(M,b)}=\norm{\mu}_{\F(M,b)}+\norm{\gamma}_{\F(M,b)}.$$
		Furthermore, by definition of $b$, we have $\norm{\mu}_{\F(M,b)}=\norm{\mu}_{\F(M,d)}$ and $\norm{\gamma}_{\F(M,b)}=\norm{\gamma}_{\F(M,d)}$, and thus 
		$$\|\mu\|_{\F(M,d)} + \|\gamma\|_{\F(M,d)}\le\norm{\mu+\gamma}_{\F(M,b)}\le \pare{1+\frac{\delta}{C}}\norm{\mu+\gamma}_{\F(M,d)}.$$
		
		To conclude the proof, it only remains to prove the claim. Note that, since $\gamma$ is a convex sum of the molecules $m_{x_{n+1}y_{n+1}} , \dots , m_{x_Ny_N}$, by \cite[Theorem 2.4]{RRZ} the sequence $(x_k ,y_k)_{k=n+1}^N$ is cyclically monotonic with respect to $d$, hence also with respect to $b$. Now let us consider $i_1, \ldots, i_m \in \{1, \ldots ,  N\}$, with $m\geq 2$ and indices from both sets $\{1,\ldots ,n\}$ and $\{n+1, \ldots , N\}$. Note that we may assume that the indices $i_1,\ldots,i_{m}$ are pairwise distinct.  Let $j_1<\ldots<j_{l}$ be such that the indices $i_{j_k}$, $1 \leq k \leq l$, are those which belong to the set $\set{n+1,\ldots,N}$, and let us write $I = \{j_1, \ldots , j_l\}$. To simplify the notation in the next estimates, we let $j_{l+1} := j_1$, $i_{m+1} := i_1$.
		Using the cyclic monotonicity of $(x_k ,y_k)_{k=n+1}^N$ and the definition of $\delta$, we first obtain:
		\begin{align*}
			\sum_{k=1}^{m} b(x_{i_k} , y_{i_k}) &= \sum_{k=1}^{l} d(x_{i_{j_k}} , y_{i_{j_k}}) + \sum_{\substack{k=1 \\ k \not\in I}}^{m} d(x_{i_k} , y_{i_k}) \\
			&\leq  \sum_{k=1}^{l} d(x_{i_{j_k}} , y_{i_{j_{k+1}}}) + \delta
		\end{align*}
		Next, using the triangle inequality, we get:
		\begin{align*}
			\sum_{k=1}^{m} b(x_{i_k} , y_{i_k}) &\leq \sum_{k=1}^{l} \left( \sum_{k'=j_k}^{j_{k+1}-1}d(x_{i_{k'}} , y_{i_{k'+1}}) + \sum_{k'=j_k+1}^{j_{k+1}-1}d(x_{i_{k'}} , y_{i_{k'}})  \right) + \delta \\
			&\leq \sum_{k=1}^m d(x_{i_k} , y_{i_{k+1}}) + \sum_{k=1}^n d(x_k , y_k) + \delta \\
			&\leq \sum_{k=1}^m d(x_{i_k} , y_{i_{k+1}}) + 2 \delta 
			\leq \sum_{k=1}^{m} b(x_{i_k} , y_{i_{k+1}}).
		\end{align*}
		This ends the proof of the claim, and the lemma.
	\end{proof}	
	
	\begin{lemma}\label{lem:conv_comb_1}
		Let $M$ be a metric space.
		Let $\mu:=\sum_k\lambda_k m_{x_ky_k}\in \F(M)$ where $(\lambda_k)\in \ell_1$ and let $\gamma\in \F(M)$. Let $r>0$ be such that for every $k$ there exists $u\in\supp(\gamma)\cup\set{0}$ such that
		\begin{equation*}
			x_k,y_k\in B(u,r).
		\end{equation*}
		Then $\norm{\gamma+\mu} \ge \frac{D-2r}{D+2r}(\|\gamma\| + \|\mu\|)$, where
		\begin{equation*}
			D=\inf \set{d(x,y):x\neq y \in \supp(\gamma)\cup\set{0}}.
		\end{equation*}
	\end{lemma}
	
	\begin{proof}
		As in the previous lemma, we may assume that $\gamma$ and $\mu$ are finitely supported, and that $\mu:=\sum_{k=1}^n\lambda_k m_{x_ky_k}$.
		The case when $r\ge\frac{D}{2}$ is trivial, thus we assume $r<\frac{D}{2}$. Without loss of generality, we may assume that $M=\supp(\gamma)\cup\set{0,x_1,y_1,\ldots,x_n,y_n}$. Then the sets $B(u,r)$, $u\in \supp(\gamma)\cup\set{0}$ cover the entire $M$. Let $f,g\in S_{\Lip_0(M)}$ be such that $\<f,\mu\>=\norm{\mu}$ and $\<g,\gamma\>=\norm{\gamma}$. Define $h:M \rightarrow\R$ by
		\begin{equation*}
			h\restricted_{B(u,r)}=g(u)-f(u)+f\restricted_{B(u,r)}
		\end{equation*}
		for every $u\in \supp(\gamma)\cup\set{0}$. Since $r<\frac{D}{2}$, the sets $B(u,r)$, $u\in \supp(\gamma)\cup\set{0}$ are pairwise disjoint and thus $h$ is well defined, and $h(0)=0$. Next we estimate the Lipschitz constant of $h$. Fix $v,v'\in M$ and let $u,u'\in \supp(\gamma)\cup\set{0}$ be such that $v\in B(u,r)$ and $v'\in B(u',r)$. If $u=u'$, then
		\begin{equation*}
			\abs{h(v)-h(v')}=\abs{f(v)-f(v')}\le d(v,v').
		\end{equation*}
		Otherwise $u\neq u'$ and thus by triangle inequality
		\begin{align*}
			d(v,v')\ge d(u,u')-d(u,v)-d(u',v')\ge D-2r,
		\end{align*}
		giving us
		\begin{align*}
			\abs{h(v)-h(v')}&=\abs{g(u)-g(u')+f(v)-f(u)+f(u')-f(v')}\\
			&\le d(u,u')+d(v,u)+d(u',v')\\
			&\le D+2r\le \frac{D+2r}{D-2r}d(v,v').
		\end{align*}
		Therefore $h\in\Lip_0(M)$ with $\norm{h}\le \frac{D+2r}{D-2r}$. On the other hand $h(u)=g(u)$ for $u\in \supp(\gamma)$ and thus $\<h,\gamma\>=\<g,\gamma\>=\norm{\gamma}$. Furthermore, for any $k\in\set{1,\ldots,n}$ there exists $u\in \supp(\gamma)\cup\set{0}$ such that $x_k,y_k\in B(u,r)$ and thus
		\begin{equation*}
			h(x_k)-h(y_k)=f(x_k)-f(y_k)
		\end{equation*}
		giving us $\<h,\mu\>=\<f,\mu\>=\norm{\mu}$. Hence
		\begin{equation*}
			\norm{\gamma+\mu}\ge\frac{1}{\norm{h}}\<h,\gamma+\mu\>\ge \frac{D-2r}{D+2r}(\|\gamma\| + \|\mu\|). 
		\end{equation*}
	\end{proof}
	
	Using the two previous lemmas, we obtain the following generalization of \cite[Theorem 2.6]{JRZ}, which covers the case where each $\gamma_n$ is a molecule.
	
	\begin{theorem}\label{pr:jrz_improved}
		Let $(\gamma_n)$ be a sequence in $S_{\lipfree{M}}$ such that each $\gamma_n$ is expressed as follows:
		$$
		\gamma_n=\sum_k \lambda_k^n m_{x_k^n y_k^n}.
		$$
		If
		$
		\lim_{n\to\infty}\sum_k d(x_k^n,y_k^n)=0,
		$
		and $\sup_n\sum_k \abs{\lambda_k^n}<\infty$,
		then $\lim_n\norm{\gamma+\gamma_n}=\norm{\gamma}+ 1$ for all $\gamma\in \lipfree{M}$. 
	\end{theorem}

	\begin{proof}
		By an easy approximation argument, we may assume without loss of generality that each $\gamma_n$ is finitely supported, and we write:
		$$ \gamma_n=\sum_{k=1}^{N_n} \lambda_k^n m_{x_k^n y_k^n}. $$
		
		Further, we may assume $\lambda_k^n>0$ by interchanging $x_k^n$ and $y_k^n$ whenever necessary.
		We may also assume that $\gamma$ is finitely supported. Fix $\varepsilon\in(0,1/3)$.
		Let 
		\begin{equation*}
			D:=\min \set{d(x,y):x\neq y \in \supp(\gamma)\cup\set{0}}.
		\end{equation*}
		Let $m\in\N$ be such that $\frac{1}{m}\sup_n\sum_k \abs{\lambda_k^n}<\varepsilon$ and let $\delta>0$ be such that $\delta\left(1+\frac{1}{\varepsilon}\right)^m +\delta< \varepsilon D$. Next we pick $n \in \N$ large enough so that 
		$$\displaystyle\sum_{k=1}^{N_n}d(x_k^n,y_k^n)<\delta.$$ Let $A_0=\emptyset$,
		$$A_i=\bigcup_{u\in\supp(\gamma)\cup\set{0}}B\left(u,\delta\left(1+\frac{1}{\varepsilon}\right)^i\right)$$
		for $i=1,\ldots,m$,  and $B_i=A_i\setminus A_{i-1}$ for $i=1,\ldots,m$. Note that the sets $B_i$ are pairwise disjoint. Hence each $x_k^n$ can belong to at most one of these sets. Thus, thanks to our choice of $m$, there must exist $i\in \set{1,\ldots,m}$ such that $\sum_{k \, :\,x_k^n\in B_{i}} \lambda_k^n <\varepsilon$. Let
		\begin{align*}
			I &= \set{k\in\set{1,\ldots,N_n}:x_k^n\notin A_i } \\
			J &= \set{k\in\set{1,\ldots,N_n}:x_k^n\in A_{i-1}}. 
		\end{align*}
		Then 
		$$\norm{\sum_{k\in I\cup J}\lambda_k^{n}m_{x_k^ny_k^n}}\ge1-\sum_{k\notin I\cup J}\lambda_k^{n} = 1-\sum_{k \, :\,x_k^n\in B_{i}} \lambda_k^n > 1-\varepsilon.$$
		Let $\mu=\sum_{k\in J}\lambda_k^n m_{x_k^n y_k^n}$. For every $k\in J$ there exists $u\in\supp(\gamma\cup\set{0})$ such that
		$$x_k^n\in  B\left(u,\delta\left(1+\frac{1}{\varepsilon}\right)^{i-1}\right)\subset B\left(u,\varepsilon D-\delta\right)$$
		and since $d(x_k^{n},y_k^{n})<\delta$, we get
		\begin{equation*}
			x_k^{n},y_k^{n}\in B(u,\varepsilon D).
		\end{equation*}
		We apply Lemma \ref{lem:conv_comb_1} to elements
		$\mu$ and $\gamma$ and obtain 
		\begin{equation*}
			\norm{\gamma+\mu} \ge \frac{D-2\varepsilon D}{D+2\varepsilon D}\left(\|\gamma\| + \norm{\mu}\right)=\frac{1-2\varepsilon}{1+2\varepsilon }\left(\|\gamma\| +  \norm{\mu}\right).
		\end{equation*}
		Now consider $C:=\frac{\delta}{\varepsilon}-2\delta>\delta$. Since for every $x\in \supp(\gamma+\mu)\cup\set{0}$ there exists $u\in \supp(\gamma)\cup\set{0}$ such that 
		$x\in B\left(u,\delta\left(1+\frac{1}{\varepsilon}\right)^{i-1}+\delta\right)$ and for every $y\in \set{x_k^n,y_k^n: k\in I}$ we have $y\notin B\left(u,\delta\left(1+\frac{1}{\varepsilon}\right)^i-\delta\right)$, we deduce that
		\begin{align*}
			d(x,y)&\ge d(u,y)-d(u,x)\ge \delta\left(1+\frac{1}{\varepsilon}\right)^i-\delta-\left(\delta\left(1+\frac{1}{\varepsilon}\right)^{i-1}+\delta\right)\\
			&=\delta\left(1+\frac{1}{\varepsilon}\right)^{i-1}\frac{1}{\varepsilon}-2\delta\ge \frac{\delta}{\varepsilon}-2\delta= C.
		\end{align*}
		Thus from Lemma \ref{lem:conv_comb} we get
		\begin{align*}
			\norm{\gamma+\mu+\sum_{k\in I} \lambda_k^n m_{x_k^n y_k^n}} &\ge \frac{C}{C+\delta}\Big(\|\gamma+\mu\| + \Big\|\sum_{k\in I} \lambda_k^n m_{x_k^n y_k^n} \Big\| \Big) \\
			&=\frac{1-2\varepsilon}{1-\varepsilon}\Big(\|\gamma+\mu\|+ \Big\|\sum_{k\in I} \lambda_k^n m_{x_k^n y_k^n} \Big\| \Big). 
		\end{align*}
		Therefore
		\begin{align*}
			\norm{\gamma +\gamma_n}&\ge \norm{\gamma+\mu+\sum_{k\in I} \lambda_k^n m_{x_k^n y_k^n}}-\norm{\sum_{k\notin I\cup J} \lambda_k^n m_{x_k^n y_k^n}}\\
			&\ge\frac{1-2\varepsilon}{1-\varepsilon}\Big(\|\gamma+\mu\|+ \Big\|\sum_{k\in I} \lambda_k^n m_{x_k^n y_k^n} \Big\| \Big)-\varepsilon\\
			&\ge \frac{(1-2\varepsilon)^2}{(1-\varepsilon)(1+2\varepsilon)}\Big(\|\gamma\|+ \norm{\mu}+ \Big\|\sum_{k\in I} \lambda_k^n m_{x_k^n y_k^n} \Big\| \Big)-\varepsilon\\
			&\ge \frac{(1-2\varepsilon)^2}{(1-\varepsilon)(1+2\varepsilon)}\Big(\|\gamma\|+1 -\varepsilon\Big)-\varepsilon.
		\end{align*}
		As $\varepsilon$ was arbitrary, we get $\lim_n\norm{\gamma+\gamma_n}=\norm{\gamma}+ 1$.
	\end{proof}
	
	This is enough to deduce the already announced main result of this section, Theorem \ref{thmD}, which corresponds to Corollary \ref{pr:sum_distances_l1_basis} below.
	We recall that a sequence $(e_n)\subset S_X$ of distinct points in the unit sphere of a Banach space $X$ is called \emph{an asymptotically isometric $\ell_1$-basis} if for every $C \in (0,1)$ there exists $k_0 \in \N$ such that for every $(a_k) \in c_{00}$ we have $\norm{\sum_{k=k_0}^\infty a_k e_k}\geq C\sum_{k=k_0}^\infty \abs{a_k}$.
	Obviously such a sequence is equivalent to the $\ell_1$-basis.
	
	\begin{corollary}
		\label{pr:sum_distances_l1_basis}
		Let $M$ be a complete metric space. For every $n \in \N$, let  $\gamma_n=\sum_k \lambda_k^n m_{x_k^n y_k^n} \in S_{\F(M)}$ such that 
		$$\sum_k d(x_k^n,y_k^n)\underset{n \to \infty}{\longrightarrow} 0\quad\text{and}\quad \sup_n\sum_k |\lambda_k^n|<\infty.$$
		Then there is an infinite $\mathbb M \subset \N$ such that
		$(\gamma_n)_{n\in \mathbb M}$ is an asymptotically isometric $\ell_1$-basis.
	\end{corollary}
	
	\begin{proof}
		Let $(\varepsilon_n)$ be a sequence of positive numbers such that $\prod_{n=1}^\infty (1-\varepsilon_n)>0$.
		A routine argument employing the conclusion of Theorem~\ref{pr:jrz_improved} allows us to choose $(\gamma_{n_k})$ iteratively so that for every $x\in \lspan\set{\gamma_{n_1},\ldots,\gamma_{n_k}}$ one has $\norm{x+a\gamma_{n_{k+1}}}\geq (1-\varepsilon_n)(\norm{x}+\abs{a})$. 
		It follows that for every $k_0$ one has
		\[
		\norm{\sum_{k=k_0}^\infty a_k\gamma_{n_k}}\geq \prod_{n=k_0}^\infty (1-\varepsilon_n)\sum_{k=k_0}^\infty \abs{a_k}.
		\]
		Since $\prod_{n=k_0}^\infty (1-\varepsilon_n) \to 1$ as $k_0\to \infty$ we get the desired conclusion.
	\end{proof}
	
	\begin{remark}\hfill
		\begin{enumerate}[leftmargin=*]
			\item Even if $(\gamma_n)$ is as in Corollary~\ref{pr:sum_distances_l1_basis}, the set $\set{\gamma_n:n\in \N}$ can still be uniformly regular.
			Indeed, denote by $(x_i^{n} , y_i^{n})\subset[0,1]$, $i=1,\ldots,2^{n-1}$ the intervals removed in the $n^\text{th}$ step of the construction of the middle-thirds Cantor set. Now, consider the sequence $(\gamma_n)\subset\F(M)$ defined as
			$$\gamma_n := \frac{1}{2^{n-1}}\sum_{i=1}^{2^{n-1}} m_{x_i^{n}y_i^{n}}.$$
			Then following the proof of Proposition~\ref{prop:PerfectURnotwpreCPT} one can show that $\set{\gamma_n:n\in \N}$ is uniformly regular.
			\item We do not know, on the other hand, whether $(\gamma_n)$ being as in Corollary~\ref{pr:sum_distances_l1_basis} implies that $\set{\gamma_n:n\in \N}$ cannot be a V*-set.
			It seems reasonable to look for a counterexample in the space $\Free(C[0,1])$ as $C[0,1]$ has no complemented copy of $\ell_1$ in it.
			\item Given $n\in \N$, let $(x_k^n) \subset [0,1]$, for $k=0,\ldots, 2^n$, be the $n$-th generation dyadic numbers, 
			and let $y_k^n=x_k^n+2^{-(n+k)}$. 
			Then $\sum_k d(x_k^n,y_k^n) \to 0$ with $n\to \infty$ but for any infinite $\M\subset \N$ we have $[0,1] \subset \overline{\bigcup_{\M}\supp \gamma_n}$. (Here the exact definition of $\gamma_n$ is inconsequential, it can be for example $\gamma_n = \frac1{2^n+1}\sum_{k=0}^{2^n} m_{x_k^n y_k^n}$.)
			This shows that there are sequences $(\gamma_n)$ as in Corollary~\ref{pr:sum_distances_l1_basis} without any subsequences living in the free space of a complete purely 1-unrectifiable metric space.
		\end{enumerate}
	\end{remark}
	
	The following corollary which provides another necessary condition for weak precompactness is strongly related to condition (v) of Proposition \ref{Char(a)_equivalent}.
	\begin{corollary}\label{c:DistanceToCSM}
		Let $M$ be a complete metric space and let $W\subset\lipfree{M}$ be weakly precompact. Then for every $\ep >0$, there exists $\delta >0$ such that 
		$$\dist\left(\gamma,\SSM{\delta}{M}\right)\ge \frac{\norm{\gamma}}{2}-\varepsilon$$ 
		for every $\gamma \in W$.
	\end{corollary}
	
	For the proof we will need the following ``big-perturbation lemma'' for $\ell_1$-bases.
	\begin{lemma}\label{l:ExceptionnalStabilityOfl1Basis}
		Let $(e_n) \subset S_X$ be $C$-equivalent to the $\ell_1$ basis, i.e. $C\sum |a_i| \leq \|\sum a_ie_i\|$. Let $\|y_n-e_n\|\leq D<C$. Then $(y_n)$ is equivalent to the $\ell_1$ basis.
	\end{lemma}
	\begin{proof}
		We have $\|\sum a_ny_n\|\geq \|\sum a_n e_n\|-\|\sum a_n(y_n-e_n)\|\geq (C-D)\sum |a_n|$ and $\|\sum a_ny_n\|\leq \sum|a_n|\|y_n\|\leq (1+D)\sum |a_n|$.
	\end{proof}
	\begin{proof}[Proof of Corollary~\ref{c:DistanceToCSM}]
		We prove the contrapositive. Assume that there exists $\varepsilon>0$ such that for every $n\in \N$, there exists $\gamma_n\in W$ such that $$\dist\left(\gamma_n,\SSM{\frac{1}{n}}{M}\right)< \frac{\norm{\gamma_n}}{2}-\varepsilon.$$
		We assume that $W$ is bounded, otherwise it is not weakly precompact and we are done. 
		Let $R=\sup_{\gamma\in W}\norm{\gamma}$.  
		Let $\mu_n\in \SSM{\frac{1}{n}}{M}$ be such that $\norm{\gamma_n-\mu_n}<\frac{\norm{\gamma_n}}{2}-\varepsilon$. 
		We show that $(\gamma_n)$ has a subsequence equivalent to the $\ell_1$ basis and thus $W$ is not weakly precompact.
		First observe that $2\varepsilon<\norm{\gamma_n}\leq R$, so we may instead find a subsequence of $(\frac{\gamma_n}{\norm{\gamma_n}})$ equivalent to the $\ell_1$-basis.
		Since $2\varepsilon<\norm{\mu_n}$ for every $n$, Corollary~\ref{pr:sum_distances_l1_basis} yields a subsequence of $(\frac{\mu_n}{\norm{\mu_n}})$ which is an asymptotically isometric $\ell_1$-basis.
		Thus if we find $D<1$ such that $\norm{\frac{\gamma_n}{\norm{\gamma_n}}-\frac{\mu_n}{\norm{\mu_n}}}\leq D$ for every $n$, Lemma~\ref{l:ExceptionnalStabilityOfl1Basis} will lead to the desired conclusion.
		But we have by the classical Massera-Sch\"affer inequality~\cite[Tapa 15]{HU}
		\[
		\norm{\frac{\gamma_n}{\norm{\gamma_n}}-\frac{\mu_n}{\norm{\mu_n}}}\leq \frac{2\norm{\gamma_n-\mu_n}}{\max\set{\norm{\gamma_n},\norm{\mu_n}}}\leq \frac{\norm{\gamma_n-\mu_n}}{\norm{\gamma_n-\mu_n}+\varepsilon}\leq 1-\frac{2\varepsilon}{R}
		\]
		so setting $D=1-\frac{2\varepsilon}{R}$ finishes the proof.
	\end{proof}

	\section{Final remarks}
	
	In line with the previous two sections, we may characterize norm compactness by imposing conditions on representations in terms of series of molecules. Roughly speaking, elements of a compact set can be approximated by series of ``large molecules'', that is, each molecule is formed by a pair of points separated by a distance of at least some fixed $\delta>0$. Before stating the result, it is worth recalling that the map $Q: (a_{(x,y)})_{(x,y) \in \widetilde{M}} \mapsto \sum_{(x,y) \in \widetilde{M}} a_{(x,y)} m_{x y}$ defines a linear quotient map from $\ell_1(\widetilde{M})$ onto $\F(M)$.
	
	\begin{proposition}\label{p:CharCompacts}
		Let $M$ be a complete metric space and $K \subset \Free(M)$ be bounded. Then the following assertions are equivalent:
		\begin{enumerate}[$(i)$]
			\item $K$ is relatively compact.
			\item There exists $c\geq 1$ such that: $\forall \varepsilon>0$, $\exists \delta>0$, $\exists C\subset M$ compact,  $\forall \mu \in K$,  $\exists a \in \ell_1(\widetilde{C})$ with $\|Qa-\mu\|<\varepsilon$, $\|a\|_1 \leq c\|\mu\|$ and $\sum_{d(x,y)<\delta} |a_{(x,y)}|<\varepsilon$.
		\end{enumerate}
	\end{proposition}
	
	\begin{proof}
		$(i) \implies (ii)$: Fix $\ep >0$ and an open cover $K \subset \bigcup_{i=1}^n B(m_i,\ep)$, with $m_i$ being finitely supported for every $i \in \{1, \ldots , n\}$. Each $m_i$ can be written as a convex sum of molecules $m_i = \sum_{k=1}^{n_i} a_{(x_k^i , y_k^i)} m_{x_k^i y_k^i}$. Let $\delta:=\underset{i,k}{\min}\, d(x_k^i , y_k^i)$. Then assertion $(ii)$ clearly follows with $c = 1$ and
		$$C = \{x_k^i : 1 \leq k \leq n_i, \, 1\leq i \leq n\} \cup \{y_k^i : 1 \leq k \leq n_i, \, 1\leq i \leq n\}.$$
		
		$(ii) \implies (i)$ We will prove that for every $\ep>0$, there exists a relatively compact subset $S \subset \F(M)$ such that $K \subset S + 2\ep B_{\F(M)}$. This readily implies that $K$ is relatively compact, as required. Let us fix $\ep> 0$. Consider $\delta>0$ and $C \subset M$ compact, as specified in $(ii)$.
		Now consider $\mu \in K$. By assumption, there exists $a \in \ell_1(\widetilde{C})$ with $\|Qa-\mu\|<\varepsilon$, $\|a\|_1 \leq c\|\mu\|$ and $\sum_{d(x,y)<\delta} |a_{(x,y)}|<\varepsilon$. Denote $e_{(x,y)} \in \ell_1(\widetilde{M})$ the element which is 0 everywhere except at the pair $(x,y)$, where it takes the value 1. We express $a$ as follows:
		$$a= \sum_{d(x,y)\geq \delta}a_{(x,y)}e_{(x,y)} + \sum_{d(x,y)<\delta}a_{(x,y)}e_{(x,y)}.$$
		If we denote $\widetilde{C}_\delta=\{(x,y) \in \widetilde{C}: d(x,y)\geq \delta\}$ then it is clear that 
		$$Qa \in Q(c\norm{\mu}\cdot B_{\ell_1(\widetilde{C}_{\delta})})+\varepsilon B_{\Free(M)}.$$ Hence
		$K \subset Q(cRB_{\ell_1(\widetilde{C}_\delta)})+2\varepsilon B_{\Free(M)}$ where $R=\max\set{\norm{\mu}:\mu\in K}$.
		Now the fact that $Q(cRB_{\ell_1(\widetilde{C}_{\delta})})$ is relatively compact follows from the next lemma.
	\end{proof}
	
	\begin{lemma} Let $C \subset M$ be compact. For $\delta>0$ denote $\widetilde{C}_\delta=\{(x,y) \in \widetilde{C}: d(x,y)\geq \delta\}$. Then $Q\restriction_{\ell_1(\widetilde{C}_\delta)}$ is a compact operator.
	\end{lemma}
	\begin{proof}
		Since $Q(B_{\ell_1(\widetilde{C}_\delta)}) \,\subset \, \overline{\conv}\{Q(e_{x,y}) : (x,y) \in \widetilde{C}_\delta\}$, it is enough to show that $\{Q(e_{(x,y)}) : (x,y) \in  \widetilde{C}_\delta\}$ is compact. Indeed, let $(x_n,y_n) \subset \widetilde{C}_\delta$. By compactness of $C$ and passing to a subsequence if necessary, we may assume that $(x_n)$ and $(y_n)$ converge to $x$ and $y$ respectively. Clearly $d(x,y) \geq \delta$ and so $Q(e_{x_n,y_n}) = m_{x_ny_n} \to m_{x y} = Q(e_{(x,y)})$.
	\end{proof}
	
	\begin{remark}
		We wish to highlight an alternative characterization of compact sets in Lipschitz free spaces, as established in \cite{LCGarcia}. To set the stage, recall Grothendieck’s compactness principle \cite[p. 112]{Grothendieck}, which states that every norm-compact subset of a Banach space is contained in the closed convex hull of a norm-null sequence. Adapting this principle to the setting of Lipschitz free spaces requires additional work, as demonstrated in \cite[Corollary~3.5]{LCGarcia} (albeit in a slightly different context). Specifically, a closed subset $K\subset \F(M)$ is compact if and only if there exist sequences $(\alpha_n)_n \in c_0$ (with $\alpha_n > 0$) and $\big((x_n,y_n)\big)_n \subset \widetilde{M}$ such that 
		$$K \subset \overline{\mathrm{conv}}\{\alpha_n m_{x_n y_n} : n \in \N\}.$$ 
		This characterization once again highlights the pivotal role of elementary molecules.
	\end{remark}

	%--------------------------------------------------------------
	\section*{Acknowledgments}

	This work was initiated while the first two named authors were invited by Professor E. Perneck\'a at the Czech Technical University in Prague in 2023. They are grateful to Professor E. Perneck\'a for valuable discussions and for providing excellent working conditions during their visit. The authors also wish to thank Professor A. Quilis for useful conversations. 
	The first and last named authors visited Universit\'e de Franche-Comt\'e/Universit\'e Marie et Louis Pasteur on several occasions, for which they wish to express their gratitude.
	Finally, the authors thank the anonymous referee for careful reading of the manuscript. 
	
	R. J. Aliaga was partially supported by Grant PID2021-122126NB-C33 funded by MICIU/AEI/10.13039/501100011033 and by ERDF/EU.
	
	This work was partially supported by the French ANR projects ANR-20-CE40-0006 and ANR-24-CE40-0892-01.
	
	T. Veeorg was supported by the Estonian Research Council grant (PRG1901).
	
	A. Proch\'azka and T. Veeorg were partially supported by the PARROT French-Estonian science and technology cooperation programme.

	% ---------------- Bibliography -----------------

\end{document}